\documentclass[12pt,leqno, oneside]{amsart}

\usepackage{amssymb,amsmath,amsthm, nccmath}
\usepackage[latin1,utf8]{inputenc}
\usepackage{graphicx,color}
\usepackage[dvipsnames]{xcolor}
\usepackage{stackrel}

\usepackage[toc,page]{appendix}
\usepackage{comment}

\usepackage[left=2cm,top=2.5cm,bottom=2.5cm,right=2cm]{geometry}
\usepackage[english]{babel}

\usepackage{tikz-cd} 
\usepackage[percent]{overpic}

\def \R {{\mathbb R}}
\def \N {{\mathbb N}}
\def \Z {{\mathbb Z}}
\def \C {{\mathbb C}}

\def \gx {\gamma(\xi)}

\def \vertspace {\rule{0pt}{2.5ex}}
\def \bx {\bar\xi}
\def \eqdef {\doteq}

\def \vertspace {\rule{0pt}{2.5ex}}

\newcommand{\lsc}[1]{(LSC)_{#1}} 
\newcommand{\tnt}{\textit{(TnT)}\ } 

\newcommand {\beqn} {\begin{equation*}}
	\newcommand {\eeqn}	{\end{equation*}}
\newcommand {\beq} {\begin{equation}}
	\newcommand {\eeq}	{\end{equation}}

\newtheorem{theorem}{Theorem}[section]
\newtheorem{lemma}[theorem]{Lemma}
\newtheorem{example}[theorem]{Example}
\newtheorem{definition}[theorem]{Definition}
\newtheorem{proposition}[theorem]{Proposition}

\newtheorem{remark}[theorem]{Remark}
\newtheorem{remarks}[theorem]{Remarks}
\newtheorem{corollary}[theorem]{Corollary}

\numberwithin{equation}{section}

\title{Chaotic dynamics in refraction galactic billiards}
\author{Vivina L. Barutello, Irene De Blasi, Susanna Terracini}

\address{University of Turin, Department of Mathematics, via Carlo Alberto 10, Turin, Italy}
\email{vivina.barutello@unito.it}
\email{irene.deblasi@unito.it}
\email{susanna.terracini@unito.it}

\date{\today} 
\thanks{Work partially supported by INdAM group G.N.A.M.P.A.}


\makeatletter
\@namedef{subjclassname@2020}{\textup{2020} Mathematics Subject Classification}
\makeatother


\begin{document}
	
\maketitle
\begin{abstract}
	We prove the presence of topological chaos at high internal energies for a new class of mechanical refraction  billiards coming from Celestial Mechanics. Given a smooth closed domain $D\in \R^2$, a central mass generates a Keplerian potential in it, while, in $\R^2\setminus \overline D$, a harmonic oscillator-type potential acts. At the interface, Snell's law of refraction holds. 
	The chaoticity result is obtained by imposing progressive assumptions on the domain, arriving to geometric conditions which holds generically in $\mathcal C^1$. The workflow starts with the existence of a symbolic dynamics and ends with the proof of topological chaos, passing through the analytic non-integrability and the presence of multiple heteroclinic connections between different equilibrium saddle points.
	This work can be considered as the final step of the investigation carried on in \cite{deblasiterraciniellissi}  and \cite{IreneSusNew}.
\end{abstract}

\section{Introduction}\label{sec: prel}
While studying dynamical systems coming from Celestial Mechanics, it is not uncommon to come across examples of chaotic models; nevertheless,   rigorously proving the chaotic nature  of a physical system is often problematic and it is, for example, the subject of the recent papers  \cite{BalGilGua2022,GuaParSeaVid2021,BarCanTer2,GuaMarSea2016,ST2012,KnTai,BolNeg,Bol,DevProgMath1981}. In this paper we propose the proof of the chaoticity of a model describing a class of mechanical refraction  billiards, which can be thought as general models for the motion of a particle subject to a discontinuous potential. Their physical interest cover many different situations (see for instance \cite{Genov2009687,Krishnamoorthy2012205}). In particular, the dynamical system considered in the present paper is of interest in Celestial Mechanics and involves two forces acting in two complementary regions of the plane: a Keplerian  center of gravitation sits inside a bounded region $D$, while a harmonic oscillator is acting in the complementary set of $D$.  This choice of the potentials appears in the literature (see \cite{Delis20152448, deblasiterraciniellissi, IreneSusNew}) to mimic the motion of a particle in an elliptic galaxy with a Black Hole in its centre; in particular in \cite{Delis20152448}, the $3$-dimensional model is considered and its chaoticity is inferred using a mixed numerical and analytical approach.
At the interface of the two regions, which we assume to be a smooth curve, a generalized Snell's refraction law holds and trajectories concatenate inner arcs (Keplerian hyperbol\ae) with outer ones (harmonic arcs) being deflected in the transition through the boundary of $D$. This indeed corresponds to some (possibly ill-defined) area-preserving map in the cylinder. While reflection billiards have been extensively investigated with and without internal potentials (as general reference we quote the monographs \cite{KozTrebook,Tabbook} and the paper \cite{takeuchi2021conformal}), at the best of our knowledge, refraction ones seem to be a new subject in the literature. 

{The main result of this paper, namely, the chaoticity of the above-said model, will be reached proceeding step by step. We start by constructing a symbolic dynamics under natural assumptions on the geometry of the boundary $\partial D$ and for high inner energies, hence we proceed with some theorems on the analytical non-integrability and the existence of multiple heteroclinic connections. Our results take advantage of the fact that a Keplerian centre acts as a scatterer at high energies (see e.g. \cite{bolotin2000periodic,Bol2017}) and complement the almost-integrability of the model proved in \cite{IreneSusNew}.}\\
The interested reader can compare our result with the well established theory of integrability of the gravitational $n$-centre problem (\cite{Bol,BolNeg,BolNeg2, BoKo, KnTai, KK1992}), while more on integrability at high energies of the $n$-center problem can be found in \cite{Kn2002,KnTai2}.

\subsection*{Analytical description of the model}\label{sec:descr_model}
Let us consider an open bounded domain $D\subset\R^2$ with a smooth boundary. The different results presented in the current paper will be obtained imposing different assumptions on $\partial D$; let us start by assuming that  $D$ contains the origin, and consider the discontinuous potential
\begin{equation*}\label{eq:potintro}
	V(z) = 
	\begin{cases} \displaystyle
		V_E(z) = \mathcal{E} -\frac{\omega^2}{2}\|z\|^2, \quad & \text{if } z \notin D, \\ &\vspace{-10pt} \\ \displaystyle
		V_I(z) = \mathcal{E}+h +\frac{\mu}{\|z\|}, \quad & \text{if } z \in D,
	\end{cases}
\end{equation*}
where $\omega^2$, $\mu$ and $\mathcal{E}$ are fixed and strictly positive, while $h$ will be chosen progressively large enough to ensure our results. We also implicitly assume that $D$ is contained in the Hill's region of the outer potential $V_E$, that is, the open ball centered at the origin and with radius $\sqrt{2\mathcal{E}}\omega^{-1}$.

The $0$-energy trajectories moving under the force induced by $V$ are concatenations of arcs in $D$, which are Keplerian hyperbol\ae\,, and arcs outside $D$, given by segments of harmonic ellipses; for the sake of clarity, we will always suppose our concatenation starting with an outer arc. The potential $V_I$ can be extended by continuity on $\partial D$, hence it makes sense to rule the transition between {\em outer} and {\em inner arcs} by means of Snell's law 
\begin{equation}\label{eq:SnellIntro}
	\sqrt{V_E(z)}\sin\alpha_E=\sqrt{V_I(z)}\sin\alpha_I, 
\end{equation} 
where $z\in \partial D$ is the transition point and  $\alpha_E$, $\alpha_I$ are the angles that the two arcs form with the normal unit vector to $\partial D$ in $z$ (see Figure \ref{fig:snell}). As we will state more precisely in Appendix \ref{sec:appB}, Eq. \eqref{eq:SnellIntro} has a variational interpretation based on a critical point argument.

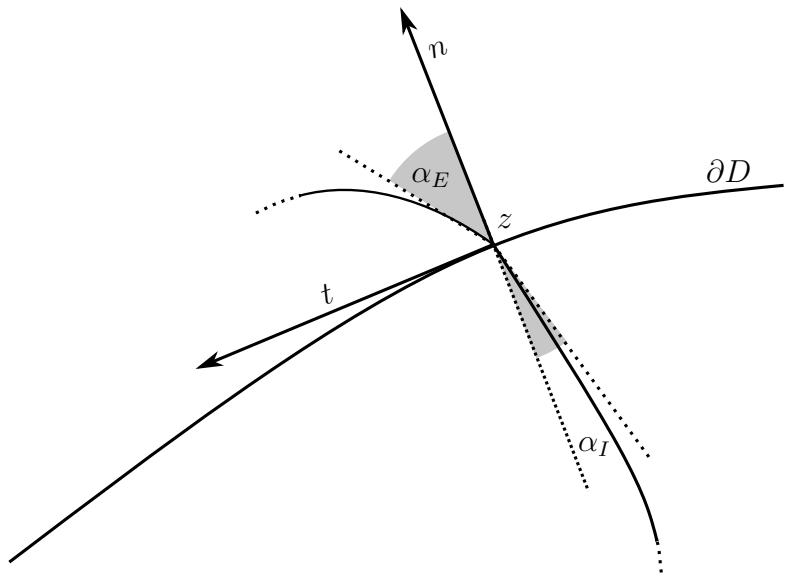
\begin{figure}
	\centering
	\begin{overpic}[width=0.6\linewidth]{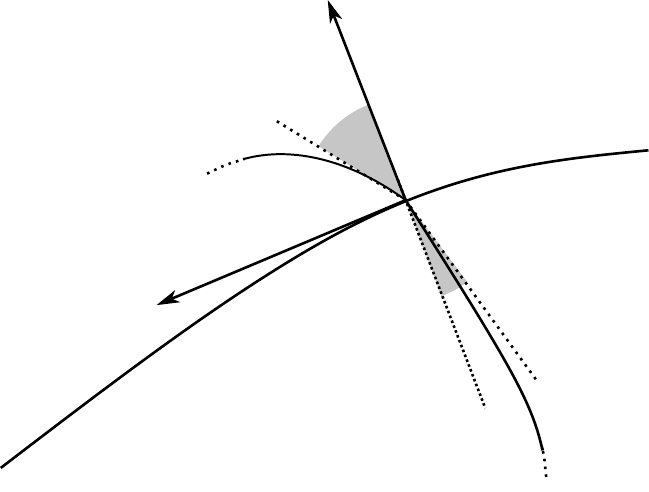}
		\put (52,53) {$\alpha_E$}
		\put (73.5,18) {$\alpha_I$}
		\put (54,69) {\rotatebox{20}{$n$}}
		\put (40,37) {\rotatebox{20}{$t$}}
		\put (90,53) {$\partial D$}
		\put (63,47) {{$z$}}
	\end{overpic}
	\caption{Snell's law between inner and outer arcs. Here, $z$ is the transition point and $t$ and  $n$ are respectively the tangent and outward-pointing normal unit vectors to $\partial D$ in $z$.}
	\label{fig:snell}
\end{figure}

Note that the stated Snell's law depends on the point $z$, and, as, for any $z\in \partial D$, $V_I(z) > V_E(z)$,  the transition from the outside to the inside is always well defined;
to pass from inside to outside we need instead to ask that the inner arc arriving at $z$ is {\em transversal enough}, namely,
\begin{equation}\label{eq:angolocritico}
	|\alpha_I| \leq \alpha_{crit} \eqdef \arcsin\left( \frac{\sqrt{V_E(z)}}{\sqrt{V_I(z)}} \right).
\end{equation}
The definition of a complete dynamics can actually be compromised from this fact. 

Special trajectories for the complete dynamics are  \emph{homothetic arcs}:   since both potentials are radially symmetric and, by Snell's law \eqref{eq:SnellIntro}, $\alpha_I =0$ if and only if $\alpha_E=0$, every radial direction hitting the boundary orthogonally is in fact preserved (recall Figure \ref{fig:figura intro2}, left). The construction of inner homothetic arcs is based on the classical Levi-Civita regularization method (see \cite{Levi-Civita} and the recent paper \cite{deblasiterraciniellissi}): the solution is in this case understood in a regularized sense; when a collision occurs it is reflected back following the same direction.
\begin{figure}
	\centering
	\begin{overpic}[width=.7\linewidth]{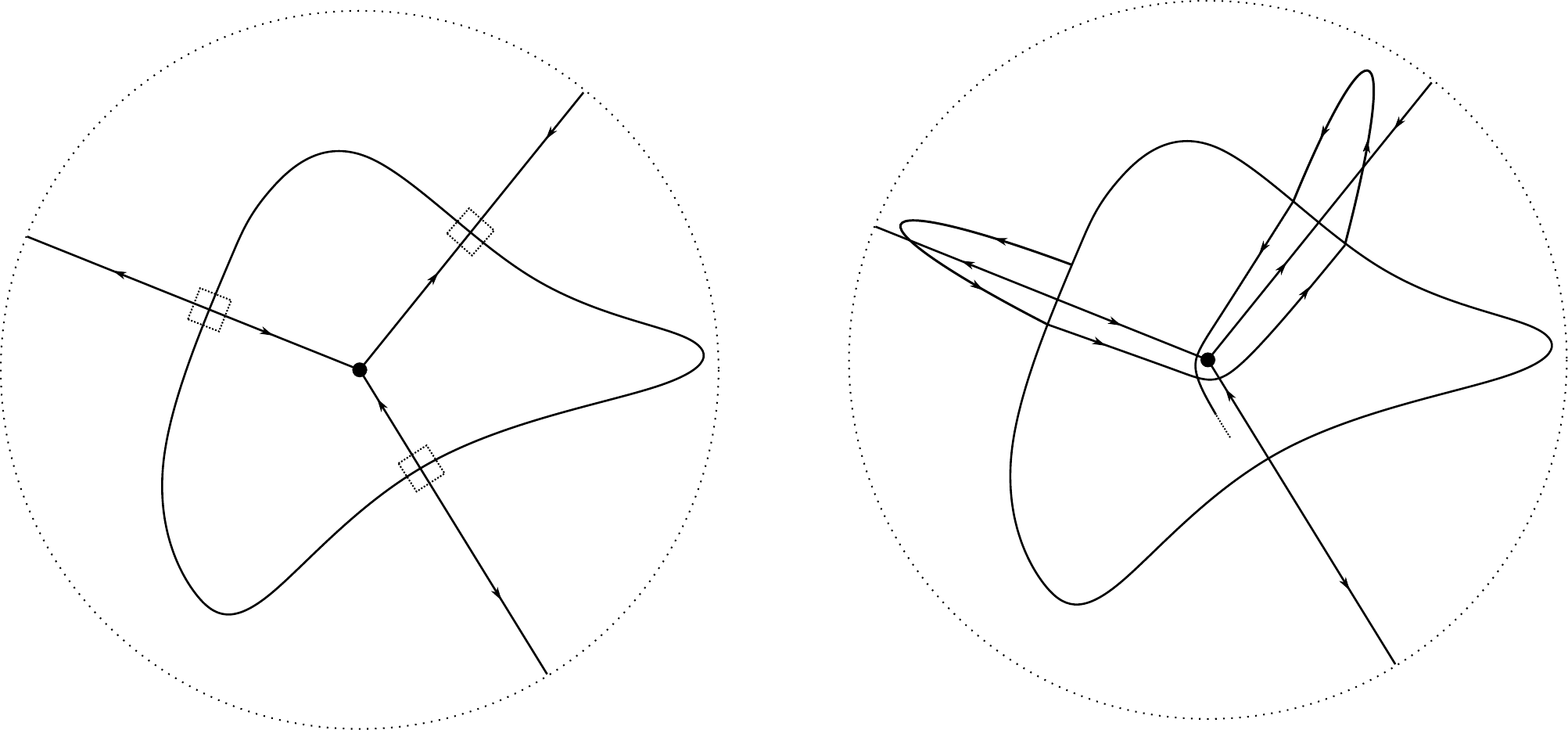}
	\end{overpic}
	\caption{Left: homothetic trajectories for the complete dynamics; the dashed circle denotes the boundary of the Hill's region for the outer potential. Right: trajectories for the complete dynamics near homothetic arcs.}
	\label{fig:figura intro2}
\end{figure}

Our aim is then to construct  outer and inner arcs connecting points in suitable regions of $\partial D$; in particular in Section \ref{sec:adm_dom_cc} such regions will be understood as neighbourhoods of points on these straight trajectories (see also Figure \ref{fig:figura intro2}, right).
The procedures to obtain such arcs are different:
\begin{itemize}
	\item outer arcs will connect points close to each others and will be  obtained via perturbative methods;
	\item inner ones will act as {\em transfer orbits} between possibly disjoint regions of $\partial D$, and can be obtained by purely geometric arguments. 
\end{itemize}

\subsection*{Statement of the results}\label{sec:statements}

{The proof of the chaoticity of our system is developed 
through intermediate results, first of them the presence of a symbolic dynamics. This is guaranteed under some dynamical and geometric conditions on the boundary of the billiard's domain $D$.} Here we give an heuristic description of these assumptions, postponing to Definitions \ref{def:prop_arco_est}, \ref{def:prop_arco_int},  and \ref{def:change_sign_prop} the rigorous explanation. The \emph{outer-arc} and \emph{inner-arc properties} regard the existence of a unique outer or inner arc connecting points in suitable subsets of $\partial D$, where we stress that the uniqueness of the inner arc is understood in a selected topological class. 
{Furthermore, the outer-arc property involves a single segment of the boundary, while, since inner arcs acts indeed as \emph{transition arcs} between different part of $\partial D$, the inner-arc property describes a union of segments of the boundary.
The \emph{change sign property} is connected to the change in sign of a partial derivative of a suitable Jacobi distance (see Appendix \ref{sec:appB}). At high energies, this property may be related  with the sign-change of the derivative of the Euclidean norm function restricted to $\partial D$.}
\begin{definition}\label{def:D_admiss_intro}
	We say that the domain $D$ is admissible if $\partial D=\gamma\left([0,L]\right)$, $\gamma\in C^1$, and there exists a finite union of disjoint intervals $\mathcal A=\cup_{i=1}^m(a_i, b_i)\subset [0,L]$, $m\geq2$, satisfying the inner-arc property (Def. \ref{def:prop_arco_int}), the change-sign property (Def. \ref{def:change_sign_prop}) and such that for every $i=i, \ldots, m$ each interval $(a_i, b_i)$ satisfies the outer-arc property (Def. \ref{def:prop_arco_est}).
\end{definition}

Now, assuming the admissibility of $D$, we can start with the construction of our symbolic dynamics. First of all we need an {\em alphabet}, given in our case by the set $\mathcal I \eqdef \{1,\ldots,m\}$ labeling the connected components of $\mathcal{A}$.
The corresponding words will be composed as bi-infinite sequences of symbols in $\mathcal I$, with a suitable grammar here specified.
\begin{definition}\label{def:adm_word_setL}
	We define the set of \emph{admissible words} for our symbolic dynamics as 
	\[
	\mathcal{L} \eqdef 
	\left\lbrace 
	\underline{\ell}=(\ell_i)_i \in \mathcal{I}^{\mathbb{Z}}\;\middle\lvert\ 
	\begin{aligned} 
		& \text{for every  $i \in \mathbb{Z}$, the symbols $\ell_i$ and $\ell_{i+1}$ do not }\\ 
		& \text{correspond to antipodal intervals}
	\end{aligned}
	\right\rbrace.
	\]
\end{definition}
In the previous definition, \emph{two intervals are antipodal} if there exists a point in the first interval and another point in the second one which are antipodally directed (see Definition \ref{def:param_antip}). The non-antipodality condition is essential for the uniqueness of the inner arc and it is not just a technical requirement. 

\begin{example}\label{ex:esempio} Let the intervals $(a_i, b_i)$ be neighbourhoods of some connected components $[\alpha_i,\beta_i]$ of the critical set of the distance function to the Keplerian center $\|\cdot\|_{|_{\partial D}}$. Let us assume that  the half-line connecting the origin to $\gamma(t)$ intersects $\partial D$ only at $\gamma(t)$ for every $t\in \cup_{i=1}^m(a_i, b_i)$. Then, the domain is admissible for large internal energies provided each interval is not self antipodal and the critical intervals are topologically stable (i.e. the derivative of the distance function changes sign close to their extremals). \end{example}

The correspondence between words and trajectories of our dynamics is given by the following definition.

\begin{definition}\label{def:realize}
	Assuming $D$ is admissible, we say that a trajectory realizes a word $\underline\ell\in \mathcal{L} $ if it visits the intervals $(a_i,b_i)$, in the order imposed by $\underline\ell$ (see Figure \ref{fig:dinsymb}). This means that there are two consecutive crossings of $\partial D$ in each $\gamma\left((a_i,b_i)\right)$.
\end{definition}
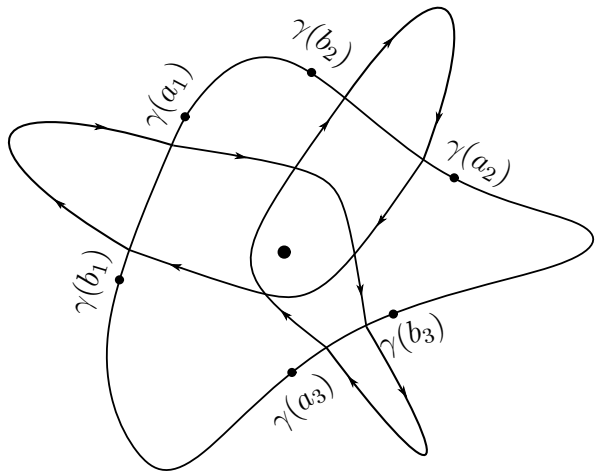
\begin{figure}
	\centering
	\begin{overpic}[width=.5\linewidth]{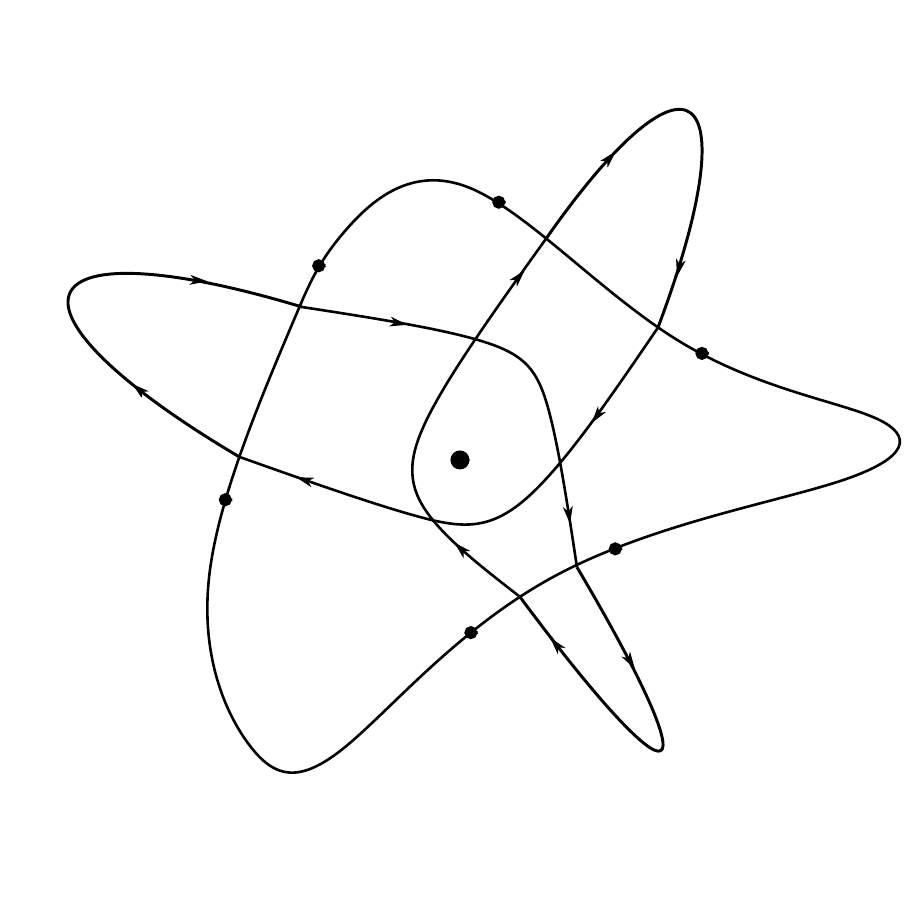}
		\put (16,40) {\rotatebox{73}{$\gamma(b_1)$}}
		\put (74,66) {\rotatebox{-30}{$\gamma(a_2)$}}
		\put (64,34) {\rotatebox{20}{$\gamma(b_3)$}}
		\put (27,69) {\rotatebox{62}{$\gamma(a_1)$}}
		\put (50,85) {\rotatebox{-35}{$\gamma(b_2)$}}
		\put (47,22) {\rotatebox{40}{$\gamma(a_3)$}}
	\end{overpic}
	\caption{Example of orbit which realizes the periodic word $\underline \ell=(\ldots, 1,3,2,1,3,2,\ldots)$. The orbit visits the three segments $\gamma\left((a_i,b_i)\right)$, $i=1,2,3$, in the order prescribed by $\underline\ell$.}
	\label{fig:dinsymb}
\end{figure}

Now we are in a position to state our first result.

\begin{theorem}\label{thm:ex_dyn_sym_intro}
	Let $D$ be an admissible domain. Then for any sufficiently large $h$ there exists a subset $X$ of the initial conditions-set, a first return map $\mathcal F$ and a continuous surjective map $\pi: X\to \mathcal L$ such that the diagram 
	\[
	\begin{tikzcd}
		X \arrow{r}{\mathcal{F}} \arrow{d}{\pi} & X \arrow{d}{\pi} \\
		\mathcal{L} \arrow{r}{\sigma_r}	& \mathcal{L}
		\end{tikzcd}
	\]
	commutes, where $\sigma_r$ is the Bernoulli right-shift. In other words, for large enough $h$, our refraction billiard model admits a symbolic dynamics.
\end{theorem}

\begin{remark}[Example \ref{ex:esempio} continued] In the setting described there, let us assume furthermore that there are at least two distinct intervals which are nor antipodal: then there are infinitely many distinct admissible words, giving rise to a nontrivial symbolic dynamics at large internal energies. The proof is an easy modification of that of Theorem \ref{thm:sym_dyn_cc}, where the connected components $[\alpha_i,\beta_i]$ are reduced to singletons. \end{remark}

To prove Theorem \ref{thm:ex_dyn_sym_intro}  we shall use a broken geodesic method, reminiscent of the one used in \cite{ST2012,BarCanTer2};
the key part of the proof is the surjectivity of the map $\pi$; in practice, this means that for every $\ell \in \mathcal L$ we need to find an initial condition in $X$ that generates a trajectory realizing the world $\ell$. This result is achieved by means of  a {\em shadowing lemma} which consists in searching critical points of a suitable length functional. Such critical points are obtained applying a fixed point theorem, known as Poincar\'e-Miranda Theorem (see \cite{Miranda}) and here the change-sign property is crucial.
In particular this techniques can be used to prove the existence of closed trajectories realizing any periodic word in $\mathcal L$.\\
The possible injectivity of $\pi$, necessary to prove the actual chaoticity of the model, may be obstructed by the lack of uniqueness of the searched critical point: we will return on this problem later.
\\
Under some restrictions on the words in $\mathcal{L}$,  we can prove that the corresponding symbolic dynamics is collision-free. In order to do that we define the set of bi-infinite symmetric words ${\mathcal L}_s \subset {\mathcal L}$ admitting a symmetry axis and state the following corollary.
\begin{corollary}\label{coro:no_coll_intro}
	Replacing ${\mathcal L}$ with $\tilde{\mathcal L} \eqdef {\mathcal L}\setminus {\mathcal L}_s$ in the diagram of Theorem \ref{thm:ex_dyn_sym_intro},
	we obtain that the symbolic dynamics is not collisional, in the sense that any trajectory corresponding to a word $\underline\ell\in\tilde{\mathcal L}$ does not have any collisional inner arc. 
\end{corollary}
Let us now come back to Theorem \ref{thm:ex_dyn_sym_intro}; the admissibility of $D$ stated in Definition \ref{def:D_admiss_intro} has a quite implicit formulation, nevertheless there are some simple geometric sufficient conditions to guarantee it. They are related to the presence of the special directions that define homothetic arcs, called, from this moment on, {\em central configurations}.

\begin{definition}\label{del:cc_intro}
	A \emph{central configuration} is a point $\overline P \in \partial D$ such that 
	\begin{itemize}
		\item $\overline P$ is a constrained critical point for $\|\cdot\|_{|_{\partial D}}$, that is, the position vector $\overline P$ is orthogonal to the boundary $\partial D$ at $\overline P$;
		\item the half-line connecting the origin to $\overline P$ intersects $\partial D$ only at $\overline P$.
	\end{itemize}
	A central configuration is termed \emph{strict} if it is a strict local maximum of minimum for $\|\cdot\|_{|_{\partial D}}$.
\end{definition}

\begin{theorem}\label{thm:sym_dyn_cc} 
	Let us suppose that $\gamma\in C^1([0,L])$ and that there exist $m$ strict central configurations, not antipodal if $m=2$. Then the domain $D$ is admissible and our refraction billiard admits a symbolic dynamics as defined in Theorem \ref{thm:ex_dyn_sym_intro}. 
\end{theorem}
In this framework, strengthening the assumptions on the central configurations we can obtain results on the analytic non integrability of the system.

\begin{definition}
\label{def:admiss_non_deg_intro}	
A central configuration $\overline P$ is termed \emph{non degenerate} if $\gamma$ is of class $C^2$ in a neighborhood of $\overline{P}$ and the second differential of the function $\|\cdot\|_{|_{\partial D}}$ is not degenerate at $\overline P$.
\end{definition}

The non-degeneracy of a central configuration allows to prove that, as far as $h$ is large enough, the corresponding homothetic equilibrium trajectory is a saddle (see Proposition \ref{prop:selle} and \cite{deblasiterraciniellissi}). The hyperbolicity of such trajectory determines the presence of a stable and of an unstable manifold which are the key objects to prove the following results.

\begin{theorem}\label{thm:no_an_int_intro}
Let us suppose that there exist $m \geq 2$ strict central configurations, not antipodal if $m=2$. Then:
\begin{itemize}
	\item assuming that one of them is non-degenerate
	and $h$ large enough, there are no analytic first integrals associated to the dynamics which are not constant;
	\item if $h$ is large enough, for every pair of distinct and non-degenerate central configurations there exist infinitely many heteroclinic connections between the corresponding homothetic trajectories.
\end{itemize}
\end{theorem}

It is a matter of fact that the results of Theorem \ref{thm:no_an_int_intro} represent a further step in the proof of the presence of chaos; in particular multiple heteroclinics are usually indicators of a complex behaviour and actually can be used to prove the presence of a symbolic dynamics 
(see \cite{BalGilGua2022,GuaParSeaVid2021,GuaMarSea2016}).
Our final aim is achieved when every central configuration is non-degenerate.

\begin{theorem}\label{thm:final_intro}
	Let us suppose that there exist $m \geq 2$ non-degenerate central configurations, not antipodal if $m=2$. Then, if $h$ is large enough, the projection map $\pi$ defined in Theorem \ref{thm:ex_dyn_sym_intro} is also injective. In other words, the dynamics of the refraction billiard admits a topologically chaotic subsystem. 
\end{theorem}

{The injectivity of the map $\pi$, or the uniqueness of the solution corresponding to a prescribed word, can also be proved, under the same non-degeneracy conditions, by means of the implicit function theorem as in \cite{bolotin2000periodic} and \cite{Bol2017}.} Furthermore, 
Theorem \ref{thm:final_intro} provides the analytical justification of the numerical simulations presented in  \cite{deblasiterraciniellissi}: here a centered elliptic domain, namely, with the singularity in its center, is taken into consideration and the presence of diffusive orbits is observed for increasing value of $h$. Of course such domain satisfies the assumptions of Theorem \ref{thm:final_intro}. \\
Let us conclude with a short digression on reflective Kepler billiards: in this dynamical models a central mass dominates the inner dynamics, but trajectories reflect elastically on the boundary. It is easy to verify that the results presented in our work apply also in this case and, actually, the dynamics is simpler since we do not have the outer part. Thus, also in this case, the dynamics associated to a centered ellipse is chaotic. This negatively complements the recent results about integrability of the focused elliptic Kepler billiard  by Takeuchi and Zhao \cite{Lei2021,takeuchi2021conformal, takeuchi2022projective}, namely,  with the singularity in one of the two foci.
This is coherent with our study, since focused ellipses, having only two central configurations which are antipodal, are not admissible domain. This provides also a useful counterexample on the importance of the non-antipodality condition, which is finally far from being just a technical assumption.

\section{Symbolic dynamics under general assumptions on $D$}\label{sec:symb_dyn}

\subsection{General assumptions on $D$}

We Let us consider an arc-length parametrization of the boundary $\partial D$ given by the function $\gamma:[0,L] \to \mathbb{R}^2$, for some $L>0$. 

\begin{definition}
	\label{def:prop_arco_est}
	An open interval $(a,b) \subset [0,L]$ satisfies the \emph{outer-arc property} if for any pair $\xi_1,\xi_2 \in (a,b)$ there exists a unique solution of the fixed-ends problem 
	\begin{equation}\label{eq:bolzaext_intro}
		\begin{cases}
			z''(s)=-\omega^2 z(s),\quad &s\in[0,T],\\
			\displaystyle	\frac{1}{2}\|z'(s)\|^2-\mathcal E+\frac{\omega^2}{2}\|z(s)\|^2=0,\quad &s\in[0,T],\\
			z(s)\notin D, &s\in(0,T),\\
			z(0)=\gamma(\xi_1), \text{ }z(T)=\gamma(\xi_2)
		\end{cases}
	\end{equation}
	for some $T \eqdef T(\xi_1,\xi_2)>0$.		
\end{definition}

Some remarks on the previous definition are due. A possible obstruction to the existence of a solution for Problem \eqref{eq:bolzaext_intro} is condition $z(s)\notin D$, $s\in(0,T)$; it turns out that such constraint is satisfied as far as the domain $D$ enjoys some simple geometric assumptions.

\begin{definition}\label{def:local_star_conv}
	We say that the domain $D$ is \emph{local star-convex with respect to $\xi\in[0, L]$}, in short $\lsc{\xi}$, if the half-line connecting the origin to $\gamma(\xi)$ intersects $\partial D$ only in $\gamma(\xi)$. Similarly, $D$ is \emph{local star-convex with respect to $A\subset [0, L]$}, in short, $\lsc{A}$, if it is $\lsc{\xi}$ for every $\xi \in A$.
\end{definition}

Choosing for instance $D=B_1(0)$, every interval $(a,b) \subset [0,2\pi]$ such that $b-a < \pi$ enjoys the outer-arc property (see also \cite[Theorem 4.1]{IreneSusNew}).
In general, we will prove in Section \ref{sec:adm_dom_cc} that when the quantity $b-a$ is small enough and $D$ is $\lsc{(a,b)}$ then $(a,b)$ satisfies the outer-arc property. 

We now proceed with analogous definitions concerning the inner dynamics. In this case we need a more complex framework to guarantee both existence and uniqueness results. 

\begin{definition}
	\label{def:param_antip}
	We say that $\xi_1,\xi_2 \in [0,L]$ are {\em antipodally directed}, or shortly, {\em antipodal}, if the origin belongs to the segment connecting $\gamma(\xi_1)$ and  $\gamma(\xi_2)$. On the other hand, two intervals $(a,b),(c,d) \subset (0,L)$ are {\em not antipodal} if every $\xi_1 \in (a,b)$ and $\xi_2 \in (c,d)$ are not antipodal. Note that $(a,b)$ and $(c,d)$ have not to be distinct.   
\end{definition}

\begin{definition}\label{def: topological char}
	Let $\xi_1, \xi_2\in[0,L]$ be not antipodal and let $\alpha:[0,A]\to \R^2\setminus\{0\}$ be a continuous curve such that $\alpha(0) = \gamma(\xi_1)$ and $\alpha(A)=\gamma(\xi_2)$. We say that $\alpha$ is {\em topologically non-trivial}, in short \tnt,  if $\alpha([0, A])$ is not homotopic to the line segment connecting $\gamma(\xi_1)$ and $\gamma(\xi_2)$ in the punctured plane $\R^2\setminus\{0\}$.  
\end{definition}

\begin{definition}
\label{def:prop_arco_int}
We say that $\mathcal{A}$ satisfies the {\em inner-arc property} if:
\begin{itemize}
	\item[\textit{(IP0)}] it is a disjoint union of open intervals
	$$
	(a_i,b_i) \subset [0,L],\quad i \in \mathcal I\eqdef\{1,\dots,m\}, \quad m \geq 2;
	$$
	\item[\textit{(IP1)}] for every $i = 1,\ldots,m$ there exists $j \neq i$ such that both intervals $(a_i,b_i)$ and $(a_j,b_j)$ are not antipodal to $(a_i,b_i)$;
	\item[\textit{(IP2)}] there exists $h_0>0$ such that for every $\xi_1,\xi_2$ belonging to two non antipodal intervals (possibly the same) and for every $h>h_0$ there exists a unique \tnt\, solution of the fixed-ends problem
	\begin{equation}\label{eq: problema interno}
	\begin{cases}
		\displaystyle	z''(s)=-\mu\frac{z(s)}{\|z(s)\|^3},\quad &s\in[0,T],\\
		\displaystyle	\frac{1}{2}\|z'(s)\|^2-\mathcal E-h-\frac{\mu}{\|z(s)\|}=0, \quad &s\in[0,T],\\
		z(s)\in D, &s\in(0,T),\\
		z(0)=\gamma(\xi_1), \text{ }z(T)=\gamma(\xi_2),
		\end{cases}
	\end{equation}
	for some $T \eqdef T(\xi_1,\xi_2;h)>0$.		
	If $\xi_1 = \xi_2$, the solution is intended in a regularized sense. 
\end{itemize}
\end{definition}
We refer to \cite[Theorem 4.2]{IreneSusNew} and \cite[Proposition 6.3]{deblasiterraciniellissi}) to provide some examples of simple domains where the existence of inner arcs is proved.\\
The local star-convexity assumption introduced in Definition \ref{def:local_star_conv} plays a role also in the inner case: if ${\mathcal{A}}$ is a disjoint union of open intervals satisfying \textit{(IP1)} and $\tilde {\mathcal{A}}$ is an open set such that $\overline {\mathcal{A}} \subset \tilde{\mathcal{A}} \subset [0,L]$ 
satisfying $\lsc{\tilde{\mathcal{A}}}$ then ${\mathcal{A}}$ satisfies \textit{(IP2)}, hence the whole inner arc property. The explicit computations will be carried out in Section \ref{sec:adm_dom_cc}.

In order to guarantee the existence of a symbolic dynamics we need a third condition on the domain $D$, which is based on the notion of Jacobi length introduced in Appendix \ref{sec:appB}, Eq. \eqref{eq:defS}. 
Given $\mathcal A$ as in Definition \ref{def:prop_arco_int}, let
\begin{equation}\label{eq:def_not_antipodal}
	NA(i)\eqdef\left\{j\in\mathcal I \text{ such that $(a_i, b_i)$ and $(a_j,b_j)$ are not antipodal}\right\}, \quad i \in \mathcal I. 
\end{equation}
By the assumptions on $\mathcal A$, $\# NA(i) \geq 2$, $i \in \mathcal I$. 

In the next definition we focus our attention on a union of compact intervals contained in $\mathcal A$ since the behaviour of our arcs at the edges of the intervals must be investigated.
\begin{definition}\label{def:change_sign_prop}
Let $\mathcal{A}$ satisfy the inner-arc property and suppose that any $(a_i,b_i)$ satisfies the outer-arc property. We say that $\mathcal{A}$ satisfies the {\em change-sign property} if there exist $h_1\geq h_0$ and $m$ compact intervals $[\alpha_i,\beta_i] \subset (a_i,b_i)$ such that for every $h>h_1$ the following inequality holds
\begin{equation*}
	\left(\partial_b S_E(\xi_E, \alpha_i)+\partial_a S_I(\alpha_i, \xi_I;h)\right)\left(\partial_b S_E(\xi_E, \beta_i)+\partial_a S_I(\beta_i, \xi_I;h)\right)<0, 
\end{equation*}
for any choice of $i \in \mathcal I$, $\xi_E\in [\alpha_i, \beta_i]$ and $\xi_I\in\bigcup_{j\in NA(i)}[\alpha_j,\beta_j]$.
\end{definition}
The change-sign property can be derived by a purely geometrical assumption {on $\partial D$} as well. In such case, the passage from the \emph{geometric} to the \emph{dynamical} property is not as straightforward as in the previous ones, and would require more sophisticated estimates. 
\begin{definition}
	Let $\mathcal{A}$ satisfy the inner-arc property and suppose that any $(a_i,b_i)$ satisfies the outer-arc property. We say that $\mathcal A$ satisfies the \emph{Euclidean change-sign} property if there exist $m$ subintervals $[\alpha_i, \beta_i]  \subset (a_i,b_i)$ such that, defined $\sigma^-_i$ (resp. $\sigma^+_i$) the angles between {$\gamma(\alpha_i)$ (resp. $\gamma(\beta_i)$) and the tangent vector to $\partial D$ at $\gamma(\alpha_i)$ (resp. $\gamma(\beta_i)$), the following inequality holds}
\begin{equation*}
	\cos\left(\sigma_i^-\right)\cos\left(\sigma_i^+\right)<0. 
\end{equation*}
\end{definition}
The geometrical meaning of the Euclidean change-sign property is the following: observing that the angles $\sigma_{i}^\pm$ are always in $[0,\pi]$, the property is verified if the two angles are in the opposite sides with respect to $\pi/2$ (see Figure \ref{fig:Euc_change_sign}). Since
\begin{equation}\label{eq:cos}
	\cos(\sigma_i^-)=\frac{d}{d\xi}\|\gamma(\xi)\|_{|_{\xi=\alpha_i}}, \quad \cos(\sigma_i^+)=\frac{d}{d\xi}\|\gamma(\xi)\|_{|_{\xi=\beta_i}},  
\end{equation}
the Euclidean change-sign property corresponds to require that the function $\xi\mapsto\|\gamma(\xi)\|$ has a change of monotonicity between $\alpha_i$ and $\beta_i$. {This property is for example verified when  there is a strict minimum or maximum for $\gamma(\cdot)$ in each $(a_i, b_i)$; this case will be investigated in Section \ref{sec:adm_dom_cc}.}

\begin{figure}
	\centering
	\begin{overpic}[width=0.5\linewidth]{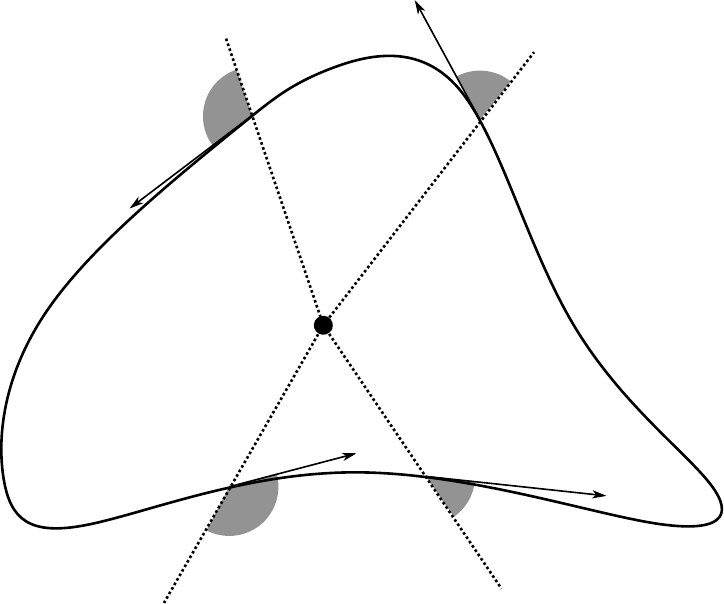}
		\put (68,64) {\rotatebox{0}{\small$\gamma(\alpha_1)$}}
		\put (65,75) {\rotatebox{0}{\small$\sigma^-_1$}}		
		\put (37,66) {\rotatebox{0}{\small$\gamma(\beta_1)$}}
		\put (22,69) {\rotatebox{0}{\small$\sigma^+_1$}}
		\put (37,10) {\rotatebox{0}{\small$\sigma^-_2$}}
		\put (20,18) {\rotatebox{0}{\small$\gamma(\alpha_2)$}}
		\put (50,12) {\rotatebox{0}{\small$\gamma(\beta_2)$}}
		\put (67,10) {\rotatebox{0}{\small$\sigma^+_2$}}
	\end{overpic}
	\caption{Geometrical meaning of the Euclidean change-sign property: at the edges of the segments of $\partial D$ the angles between the tangent vector and the radial direction are opposite w.r.t. $\pi/2$.}
	\label{fig:Euc_change_sign}
\end{figure}
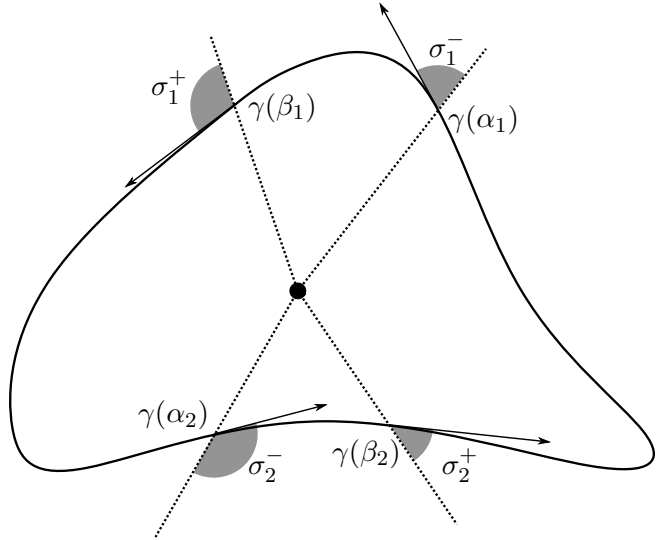

\begin{proposition}\label{prop:cs_geom}
	If $\mathcal A$ satisfies the Euclidean change-sign property, then it satisfies the change-sign property as well. 	
\end{proposition}
\begin{proof}
{Let us start by showing that the quantity $\partial_b S_E(\xi_1, \xi_2)$, $\xi_1,\xi_2\in[\alpha_i, \beta_i]$, $i \in \mathcal I$, is bounded uniformly in the endpoints. This is a simple consequence of Lemma \ref{lem:derivateS} from which we deduce the existence of a positive constant $C$ such that, for every $\xi_1, \xi_2\in [\alpha_i,\beta_i]$, $i\in \mathcal I$, 
\begin{equation*}
	\vert\partial_b S_E(\xi_1,\xi_2)\vert=\sqrt{V_E\left(\gamma(\xi_2)\right)}\Big\vert\cos\left(\angle\left(z'_E(T; \gamma(\xi_1),\gamma(\xi_2)), \dot\gamma(\xi_2)\right)\right)\Big\vert\leq C, 
\end{equation*}
where $z_E(\cdot; \gamma(\xi_1), \gamma(\xi_2))$ is the outer arc connecting $\gamma(\xi_1)$ to $\gamma(\xi_2)$. \\
Let us now fix $i \in \mathcal I$: by the Euclidean change-sign property and Eq. \eqref{eq:cos}, one has that 
\begin{equation*}
	\left(\frac{d}{d\xi}\|\gamma(\xi)\|_{|_{\xi=\alpha_i}}\right)\left(\frac{d}{d\xi}\|\gamma(\xi)\|_{|_{\xi=\beta_i}}\right) < 0;   
\end{equation*}
to fix the ideas, let us suppose
\begin{equation}\label{eq:segni_Euclidei}
	\frac{d}{d\xi}\|\gamma(\xi)\|_{|_{\xi=\alpha_i}}>0 \quad\text{and}\quad \frac{d}{d\xi}\|\gamma(\xi)\|_{|_{\xi=\beta_i}}<0.   
\end{equation}
Take now $\xi_E\in [\alpha_i, \beta_i]$, $\xi_I\in\bigcup_{j\in NA(i)}[\alpha_j,\beta_j]$, and suppose $h$ to be greater than $h_0$: we can then define the quantities $S_E(\xi_E, \alpha_i)+S_I(\alpha_i, \xi_I;h)$ and  $S_E(\xi_E, \beta_i)+S_I(\beta_i, \xi_I;h)$, corresponding to the total Jacobi lengths of the two concatenation outer-inner arcs connecting $\gamma(\xi_E)$ to $\gamma(\xi_I)$ and passing respectively through $\gamma(\alpha_i)$ and $\gamma(\beta_i)$. We want to show that, if $h$ is sufficiently large, 
\begin{equation}\label{eq:tesi_cs}
	\begin{aligned}
	&\partial_b S_E(\xi_E, \alpha_i)+\partial_a S_I(\alpha_i, \xi_I;h)>0\\
	&\partial_b S_E(\xi_E, \beta_i)+\partial_a S_I(\beta_i, \xi_I;h)<0.  
	\end{aligned}
\end{equation}
From Lemma \ref{lem:lemma_sus}, one has that 
\begin{equation*}
	S_I(\alpha_i, \xi_I;h)=\sqrt{\mathcal E+h}\left(\|\gamma(\alpha_i)\|+\|\gamma(\xi_I)\|\right) +\frac{\mu}{\sqrt{\mathcal E+h}}\left(F_1(\alpha_i,\xi_I; \mathcal E+h)-\log\left(\frac{\mu}{2(\mathcal E+h)}\right)\right), 
\end{equation*}
where, by the chain rule, $F_1$ is at least $C^1$-bounded uniformly in the first two variables when $h\to\infty$. One has then that 
\begin{equation*}
	\partial_b S_E(\xi_E, \alpha_i)+\partial_a S_I(\alpha_i, \xi_I;h)\geq -C+\sqrt{\mathcal E+h}\frac{d}{d\xi}\|\gamma(\xi)\|_{\xi=\alpha_i}+\frac{\mu}{\sqrt{\mathcal E+h}}\partial_a F_1(\alpha_i, \xi_I; \mathcal E+h); 
\end{equation*}
by virtue of Eq. \eqref{eq:segni_Euclidei}, if $h$ is sufficiently large then the above quantity is positive. With a completely analogous reasoning, one can prove also the second inequality in \eqref{eq:tesi_cs}. \\
We stress that, due to the uniform boundedness of $F_1$ with respect to the endpoints' parameters, the regularity of $\gamma$ and the compactness of $[\alpha_i,\beta_i]$, $i=1, \ldots, m$, one can find a uniform threshold value $h_1$ such that the change-sign property holds. 
}
\end{proof}

\begin{remark}\label{rem:dominant_part}
	From the proof above, one immediately understands that, as far as $h$ is large, the dominant quantity in the derivatives of the total Jacobi length is the one related to the inner arc. In general, this fact holds also for the total length itself and not only for its derivatives. 
\end{remark}

\subsection{Existence of suitable periodic trajectories}
\label{ssec:ex_per_traj}

The outer-arc, inner-arc and change-sign properties are the keystones to prove the existence of a symbolic dynamics for our refraction billiard; for this reason we give the following definition.
\begin{definition}\label{def:D_admiss}
	We say that a domain $D$ is {\em admissible} if $\partial D=\gamma\left([0,L]\right)$, $\gamma\in C^1$, and there exists $\mathcal A \subset [0,L]$ satisfying the inner-arc property, the change-sign property and such that every interval $(a_i, b_i)$ satisfies the outer-arc property.
\end{definition}
The first step to construct a symbolic dynamics is to define an {\em alphabet}, as well as the rules to build {\em admissible words} (see \cite{Dev_book}).
Take then an admissible domain $D$ and let
\begin{equation}\label{eq:def_intervalli}
I^{(i)}\eqdef[\alpha_i,\beta_i], \qquad i \in \mathcal{I},
\end{equation}
where $[\alpha_i,\beta_i]$ have been introduced in Definition \ref{def:change_sign_prop}. 

Let now $n\in \N$, $n\geq1$, and define the set $\mathbb L_n\subset\mathcal I^n$ of the admissible words of length $n$ as 
\begin{equation*}
	\mathbb L_n \eqdef \left\{\underline\ell=\left(\ell_0, \dots, \ell_{n-1}\right) : 	\ell_i\in\mathcal I\text{ and }  \ell_{(i+1)mod\,n}\in NA\left(\ell_i\right), \; i=0, \dots, n-1\right\}, 
\end{equation*}
where the sets $NA\left(\ell_i\right)$ have been introduced in Eq. \eqref{eq:def_not_antipodal}.
Since for any $i \in \mathcal{I}$ the set of not antipodal indices $NA(i)$ is not empty, the set $\mathbb L_n$ is not empty as well.
We are now ready to define the set of {\em admissible finite words}
\begin{equation}\label{eq:LL}
	\mathbb L \eqdef \bigcup_{n\in\N \setminus\{0\}}\mathbb L_n. 
\end{equation}
For any fixed $\underline \ell\in\mathbb L$ we define the $(2n+1)-$dimensional open rectangle
\begin{equation}\label{eq:defU}
	\mathbb{U}_{\underline{\ell}} \eqdef (I_0\times I_0)\times(I_1\times I_1)\times\dots\times(I_{n-1}\times I_{n-1})\times I_0, \quad 
	I_k \eqdef I^{(\ell_k)},
\end{equation}
domain of (the parameters of) the transition points between the inner and outer arcs of the periodic solutions we are searching for. To this purpose, define the closed set
\begin{equation}\label{eq:defSl}
	\mathbb S_{\underline{\ell}} \eqdef \left\{\underline{\xi}=(\xi_0, \dots, \xi_{2n})\in {\mathbb{U}}_{\underline{\ell}}\text{ : }\xi_0=\xi_{2n}\right\}.
\end{equation}
Let us point out that for every $j=0, \dots, n-1$ the points $\xi_{2j}$ and $\xi_{2j+1}$ belong to the same interval $I_j$, while the points $\xi_{2j+1}$ and $\xi_{2j+2}$ belong to possibly different intervals $I_j$ and $I_{(j+1)mod\,n}$ and correspond to non-antipodal points in $\R^2$. 
Since our complete dynamics starts with an outer arc, by virtue of Definitions \ref{def:prop_arco_est} and \ref{def:prop_arco_int} we have that for every $j=0, \dots, n-1$ and $h>0$ sufficiently large
\begin{equation}\label{eq:archiEI}
	\begin{aligned}
		&\exists ! \;z_E(\cdot; \gamma(\xi_{2j}), \gamma(\xi_{2j+1})) && \text{outer arc from $\gamma(\xi_{2j})$ to $\gamma(\xi_{2j+1})$},\\
		&\exists ! \; z_I(\cdot; \gamma(\xi_{2j+1}), \gamma(\xi_{2j+2}); h) && \text{inner {\it (TnT)} arc from $\gamma(\xi_{2j+1})$ to $\gamma(\xi_{2j+2})$ with energy $h$}.
	\end{aligned} 
\end{equation}
The relation between sequences in $\mathbb{S}_{\underline{\ell}}$ and periodic trajectories for the complete dynamics can be built as follows: given $\underline\xi\in\mathbb S_{\underline\ell}$, the corresponding periodic orbit $z(\cdot)$ is the concatenation of the arcs in Eq. \eqref{eq:archiEI}, which is unique. More precisely, one can give the following definition. 
\begin{definition}\label{def:concatenazione}
	Given $n\geq 1$, $\underline \ell\in\mathbb L_n$, $\underline \xi\in \mathbb S_{\underline \ell}$ and $h>h_0$, let us consider the unique arcs listed in Eq. \eqref{eq:archiEI} which connect pairs of subsequent points; with reference to Definitions \ref{def:prop_arco_est} and \ref{def:prop_arco_int}, let us define
	\begin{equation*}
		T_E^{(j)}\eqdef T_E\left(\xi_{2j}, \xi_{2j+1}\right), \quad  T_I^{(j)}\eqdef T_I\left(\xi_{2j+1}, \xi_{2j+2}; h\right), \quad j=0, \dots, n-1. 
	\end{equation*}
	and the partial sums $T^{(j)}\eqdef \sum_{k=0}^j T_E^{(k)}+T_I^{(k)}$. Setting $T^{(-1)}\eqdef0$, consider the concatenation $z\left(\cdot; \underline{\xi}; h\right): \left[0, T^{(n-1)}\right]\to\R^2$ where, for every $j=0, \dots, n-1$, 
	\begin{equation*}
		\begin{cases}
			z\left(s;\underline{\xi}; h\right)\eqdef z_E\left(s-T^{(j-1)}; \gamma\left(\xi_{2j}\right), \gamma\left(\xi_{2j+1}\right)\right)\quad &s\in\left[T^{(j-1)}, T^{(j-1)}+T_E^{(j)}\right]\\ & \\
			z\left(s;\underline{\xi}; h\right)\eqdef z_I\left(s-T^{(j-1)}-T_E^{(j)}; \gamma\left(\xi_{2j+1}\right), \gamma\left(\xi_{2j+2}\right); h\right)\quad &s\in\left[T^{(j-1)}+T_E^{(j)}, T^{(j)}\right] 
		\end{cases}.
	\end{equation*}
	The function $z\left(\cdot;\underline{\xi}; h\right)$ is trivially continuous, and, since $z\left(0;\underline{\xi}; h\right)=z\left(T^{(n-1)};\underline{\xi}; h\right)$, it is also periodic. We can then extend it by periodicity and, with an abuse of notation, suppose $z\left(\cdot;\underline{\xi}; h\right):\R\to \R^2$. 
\end{definition}

We stress that in general such concatenations are not $C^1$.
For $z$ to be an admissible trajectory for the complete dynamics, the Snell's law must be satisfied at every transition point $P_i\eqdef \gamma(\xi_i)$, $i=0, \dots, 2n$, namely, for every $j=0, \dots, n-1$,
\begin{equation}\label{eq:SnellConcatenation}
\footnotesize{	\begin{aligned}
		\sqrt{V_E\left(P_{2j+1}\right)}\frac{z'_E\left(T_E^{(j)}; P_{2j}, P_{2j+1}\right)}{\rule{0pt}{3.5ex}\Big\|z'_E\left(T_E^{(j)}; P_{2j}, P_{2j+1}\right)\Big\|}\cdot \dot\gamma(\xi_{2j+1})=\sqrt{V_I\left(P_{2j+1}\right)}\frac{z'_I\left(0; P_{2j+1}, P_{2j+2}; h\right)}{\rule{0pt}{3.5ex}\Big\|z'_I\left(0; P_{2j+1}, P_{2j+2}; h\right)\Big\|}\cdot \dot\gamma(\xi_{2j+1}), \\
		\sqrt{V_I\left(P_{2j+2}\right)}\frac{z'_I\left(T_I^{(j)}; P_{2j+1}, P_{2j+2}; h\right)}{\rule{0pt}{3.5ex}\Big\|z'_I\left(T_I^{(j)}; P_{2j+1}, P_{2j+2}; h\right)\Big\|}\cdot \dot\gamma(\xi_{2j+2})=\sqrt{V_E\left(P_{2j+2}\right)}\frac{z'_E\left(0; P_{2j+2}, P_{2j+3}\right)}{\rule{0pt}{3.5ex}\Big\|z'_E\left(0; P_{2j+2}, P_{2j+3}\right)\Big\|}\cdot \dot\gamma(\xi_{2j+2}),
	\end{aligned}}
\end{equation}
where, with an abuse of notation, $P_{2n+1}=P_1$. 
Since, fixed $h$ and $\underline\xi$, the concatenation $z\left(\cdot; \underline\xi; h\right)$ is uniquely determined, the validity of conditions $\eqref{eq:SnellConcatenation}$ depends only on the transition points and on the energy. 

Fixed $h>h_0$ and $\underline{\ell} \in \mathbb{L}$, the \emph{total Jacobi length} of $z\left( \cdot;\underline\xi;h \right)$ is the function
$W_{\underline{\ell}}\left(\cdot; h\right): \mathbb S_{\underline{\ell}}\to \R$ defined as 
	\begin{equation}\label{eq:azione}
		W_{\underline{\ell}}(\underline{\xi};h)=\sum_{j=0}^{n-1}S_E(\xi_{2j}, \xi_{2j+1})+\sum_{j=0}^{n-1}S_I(\xi_{2j+1}, \xi_{2j+2}; h),
		\quad n = |\underline{\ell}|,
	\end{equation}
where $S_E$ and $S_I$ has been introduced in Eq. \eqref{eq:defS}.
By means of Eq. \eqref{eq: derivate S}, we have that $\hat{\underline\xi}$ solves $\nabla W_{\underline{\ell}}\left(\hat{\underline\xi}; h\right)=\underline 0$ if and only if $z\left( \cdot;\hat{\underline\xi};h \right)$ satisfies the Snell's law at every transition point. The searches for critical points of $W_{\underline{\ell}}(\cdot; h)$ and for periodic trajectories are then equivalent. 

Provided that $h$ is large enough, the existence of a critical point of $W_{\underline\ell}(\cdot; h)$ is a straightforward consequence of the following classical result. 
\begin{theorem}[Poincar\'e-Miranda Theorem, \cite{Miranda}]\label{thm: miranda}
	Let $F_1, \dots, F_d$  $d$-functions in the variables $(x_1, \dots, x_d)$ continuous on the $d$-dimensional hypercube 	
	\begin{equation*}
		R=\left\{(x_1, \dots, x_d)\,|\,|x_k|\leq L\text{ for every  }k=1, \dots, d\right\}
	\end{equation*}
	and such that for every $k=1, \dots, d$
	\begin{equation}\label{eq: miranda disuguaglianze}
		F_k\left(x_1, \dots, x_d\right)_{|x_k=-L} \cdot
		F_k\left(x_1, \dots, x_d\right)_{|x_k=L}< 0. 
	\end{equation}
	Then there exists at least a solution in $\mathring{R}$ of
	\begin{equation}\label{eq: miranda sistema}
		F_k(x_1, \dots, x_d)=0 \quad\text{for every  }k=1, \dots, d. 
	\end{equation}
\end{theorem} 
\begin{proposition}\label{prop: punto critico azione}
	Given $\underline \ell\in\mathbb L$, the total Jacobi length $W_{\underline\ell}(\cdot; h)$ admits a critical point $\hat{\underline\xi}$ in $\mathring{S}_{\underline{\ell}}$  provided $h>h_1$, where $h_1$ is introduced in Definition \ref{def:change_sign_prop}. 
\end{proposition} 
\begin{proof}
	This proof relies on a direct application of Poincar\'e-Miranda Theorem: let us fix $\underline\ell\in\mathbb L$ and set $d\eqdef 2|\underline\ell|=2n$ and $R \eqdef \Pi_{i=0}^{n-1}\left({I_i}\times{I}_i\right)$, where ${I}_i =[\alpha_{\ell_i},\beta_{\ell_i}]$. For every $\underline\xi\in\mathbb S_{\underline\ell}$, define 
	\begin{equation*}
		\begin{aligned}
			&F_{2i}\left(\underline{\xi}\right)=\partial_b S_I\left(\xi_{2i-1}, \xi_{2i}; h\right)+\partial_aS_E\left(\xi_{2i}, \xi_{2i+1}\right), \\
			&F_{2i+1}\left(\underline{\xi}\right)=\partial_b S_E\left(\xi_{2i}, \xi_{2i+1}\right)+\partial_aS_I\left(\xi_{2i+1}, \xi_{2i+2}; h\right), \quad  i=0, \dots, n-1
		\end{aligned}
	\end{equation*}	
	where $\xi_{-1}=\xi_{2n-1}$. Computing the above functions on the hypercube's edges, one has that
	\begin{equation}
		\begin{aligned}
			F_{2i}\left(\underline\xi\right)_{|\xi_{2i}=\alpha_{\ell_i}}=\partial_b S_I\left(\xi_{2i-1},\alpha_{\ell_i} \right)+\partial_a S_E\left(\alpha_{\ell_i}, \xi_{2i+1}\right), \\
			F_{2i+1}\left(\underline\xi\right)_{|\xi_{2i+1}=\alpha_{\ell_i}}=\partial_b S_E\left(\xi_{2i},\alpha_{\ell_i} \right)+\partial_a S_I\left(\alpha_{\ell_i}, \xi_{2i+2}\right),
		\end{aligned}
	\end{equation}
	and similarly for the right edges.
	If $h>h_1$, the change-sign property ensures the application of Poincar\'e-Miranda Theorem to obtain $\hat{\underline\xi}\in\mathring{\mathbb S}_{\underline\ell}$ such that $\nabla W_{\underline\ell}\left(\hat{\underline\xi}; h\right)=0$.
\end{proof}
The main result of this section now follows straightforwardly. 
\begin{theorem}\label{thm:SymDyn}
	Given $\underline\ell\in\mathbb L$ and $h>h_1$, there exists $\hat{\underline\xi}\in\mathring{\mathbb S}_{\underline\ell}$ and a periodic trajectory $z\left(\cdot; \hat{\underline\xi}; h\right)$ for the complete dynamics which realizes the word $\underline\ell$, in the sense that it connects $\gamma\left(\hat\xi_0\right), \dots, \gamma\left(\hat\xi_{2n-1}\right)$. 
\end{theorem}

\subsection{Existence of suitable fixed-ends trajectories}\label{ssec:fixed_ends}

The same techniques used in Section \ref{ssec:ex_per_traj} to find periodic solutions can be applied to the construction of fixed-ends trajectories realizing admissible words. Although very similar to Theorem \ref{thm:SymDyn}, the final result of the current section will be of crucial importance in Section \ref{sec:non_int_chaos}, where the analytic non-integrability of our model is proved. \\
To begin our construction, let us consider some slight modifications of the sets $\mathbb L_n$ and $\mathbb S_{\underline\ell}$: in particular, for any $n\geq 2$, let 
\begin{equation*}
	\begin{aligned}
		&\mathbb L'_n=\left\{\underline\ell=\left(\ell_0, \dots, \ell_{n-1}\right) : 	\ell_i\in\mathcal I\text{ and } \ell_{i+1}\in NA\left(\ell_i\right),\; i=0, \dots, n-2\right\}, \\
		&\mathbb{U'}_{\underline{\ell}}=(I_0\times I_0)\times(I_1\times I_1)\times\dots\times(I_{n-2}\times I_{n-2})\times I_{n-1}, \\
		&\mathbb S'_{\underline{\ell}}=\left\{\underline{\xi}=(\xi_0, \dots, \xi_{2n-2})\in{\mathbb{U}}_{\underline{\ell}}\text{ : }\xi_0= \xi_a, \,\xi_{2n-2}=\xi_b\right\},
	\end{aligned}
\end{equation*}
where $\xi_a$ ad $\xi_b$ are fixed respectively in $I_0$ and $I_{n-1}$. It is easy to prove that 
\begin{equation*}
	\mathbb L \subseteq \mathbb L' = \bigcup_{n\in\N \setminus\{0\}}\mathbb L'_n,  
\end{equation*}
with $\mathbb L$ defined as in Eq. \eqref{eq:LL}. \\
Let us remark that, in this setting, we have $2n-3$ free points $\xi_1, \ldots,\xi_{2n-3}$. As a consequence, given a sequence $\underline \xi\in\mathbb S'_{\underline \ell}$, the total Jacobi length takes the form
\begin{equation}\label{eq:azione_fixed}
	W'_{\underline{\ell}}(\underline{\xi};h)=\sum_{j=0}^{n-2}S_E(\xi_{2j}, \xi_{2j+1})+\sum_{j=0}^{n-2}S_I(\xi_{2j+1}, \xi_{2j+2}; h). 
\end{equation}
In this framework, Definition \ref{def:concatenazione} and Eq. \eqref{eq:SnellConcatenation} are the same, with straightforward modifications in the indices. Again, critical points of the total Jacobi length correspond to admissible trajectories for our fixed-ends dynamics.
\begin{proposition}\label{prop:fixed_ends}
Let $\underline\ell\in\mathbb L$, $h > h_1$, $\xi_a\in I_0$ and $\xi_b\in I_{n-1}$ with $n\eqdef|\underline \ell|$.  Then there exists a trajectory which connects $\gamma\left(\xi_a\right)$ to $\gamma\left(\xi_b\right)$ realizing the word $\underline \ell$. 
\end{proposition}
\begin{proof}
	The proof is completely analogous of the one of Theorem \ref{thm:SymDyn}, setting $d\eqdef2n-3$, 
	\begin{equation*}
		\begin{aligned}
			&R\eqdef I_0\times \left(I_1\times I_1\right)\times \ldots \times \left(I_{n-2}\times I_{n-2}\right),
			\quad F_i(\underline\xi)\eqdef \frac{\partial W'_{\underline\ell}}{\partial\xi_i }\left(\underline\xi;h\right)
		\end{aligned}
	\end{equation*}
	for every $i=1, \ldots, 2n-3$. 
\end{proof}

\subsection{Construction of the symbolic dynamics}\label{ssec:dinamicaSimbolica}
We are finally ready to prove the existence of a symbolic dynamics for our refractive billiard model; in particular, we will construct a surjective and continuous application between a suitable set of initial conditions of trajectories and admissible bi-infinite words.
To this end, let us define the energy shell for the external dynamics
\[
\Xi \eqdef 
\left\{ 
(\xi,v) \colon \xi \in [0,L], v \in \R^2, \frac12 \|v\|^2 -V_E(\gx)=0 
\right\}.
\]
We now introduce the sets of initial conditions in $\Xi$ for which
the parameter $\xi$ belongs to a fixed ${I^{(r)}}$ and
the velocity vector points respectively outward or inward the domain $D$, namely,
\[
\Xi^+_r \eqdef 
\left\{ 
(\xi,v) \in \Xi \colon \xi \in {I^{(r)}} \; \text{and} \;  \langle n(\xi),v\rangle>0 
\right\},
\qquad 
\Xi^-_r \eqdef 
\left\{ 
(\xi,v) \in \Xi \colon \xi \in {I^{(r)}} \; \text{and} \;  \langle n(\xi),v\rangle<0 
\right\},
\]   
where $n(\xi)$ is the outward-pointing normal unit vector to $\gamma$ in $\gamma(\xi)$. \\
Since in our model a crossing through $\gamma$ implies a refraction of the trajectory, it is convenient to define analytically a {\em refraction map} which of course depends on the parameter $\xi$ and the energy jump $h$. From now on we assume $h> h_1$, where $ h_1$ is the threshold value introduced in Definition \ref{def:change_sign_prop} and used in Theorem \ref{thm:SymDyn}.
\begin{definition}
	Fixed $h$ and $\xi \in [0,L]$, we define the sets 
	\[
	B_E^\xi \eqdef \left\{ v \in \R^2 \colon \|v\|^2 = 2V_E(\gx) \right\}
	\qquad
	B_I^\xi \eqdef \left\{ v \in \R^2 \colon \|v\|^2 = 2V_I(\gx) \right\}
	\]
	and the refraction map
	\[
	\begin{split}
		R_{EI}(\cdot;\xi,h) \colon B_E^\xi & \to B_I^\xi \\
		v = a \,t(\xi) + b \, n(\xi) & \mapsto R_{EI}(v;\xi,h) = a\,t(\xi) + \mathop{sgn}(b) \sqrt{2V_I(\gx)-a^2}\,n(\xi)
	\end{split}
	\]
	where we recall that  $t(\xi)$ is the tangent unit vector to $\gamma$ in $\gamma(\xi)$.	
\end{definition}

\begin{remark}
	Some words are due to understand the previous definition. First of all, we observe that for every $\xi \in [0,L]$ the vectors $(t(\xi),n(\xi))$ form a orthonormal basis of the plane. Moreover, since $v \in B_E^\xi$ (namely, $a^2+b^2 = 2V_E(\gx)$) and $V_I > V_E$ everywhere in the punctured plane, the quantity $2V_I(\gx)-a^2$ is always positive, so that the refraction map is well defined in its domain. Furthermore, it is clear that the normal component preserves the sign, so that the image of an outward (resp. inward)-pointing vector is still outward (resp. inward)-pointing. It can also be easily proved that given a vector $v \in B_E^\xi$, its image $R_{EI}(v;\xi,h)$ is the unique vector with which $v$ satisfies the refraction Snell's law \eqref{eq:SnellIntro}.
\end{remark}

The refraction map $R_{EI}$ is clearly injective, hence invertible on its image
\[
\tilde B_I^\xi \eqdef \left\{ v \in \R^2 \colon \|v\|^2 = 2V_I(\gx) \text{ and } \langle v,t(\xi)\rangle \leq \sqrt{2V_E(\gx)} \right\}.
\]

\begin{remark}\label{rem:crit_value}
In view of Eq. \eqref{eq:angolocritico}, to take $ v \in \tilde B_I^\xi$ corresponds to require that the angle between $v$ and $n(\xi)$ is less than the critical value $\alpha_{crit}$.
\end{remark}

Let us now fix $(\xi,v) \in \Xi^+_r$ and follow step by step the trajectory of the complete dynamics starting from the initial condition $(\gx,v)$. Assumptions on the allowed  initial conditions will become more and more restrictive as the dynamics 
proceeds in order to obtain a final set 
\begin{equation} \label{eq:setX}
	X \subseteq \bigcup_{r \in \mathcal{I}} \Xi^+_r
\end{equation}
for which it is possible to construct a symbolic dynamics.

First of all, let us consider the flow $\Phi^s_E (\gx,v)$ generated by the Cauchy problem associated to the outer potential. 
As customary dealing with this kind of systems, we consider the projections of such flow onto the configuration and the velocity space respectively:
\[
\Pi_z \Phi^s_E (\cdot,\cdot),
\qquad 
\Pi_v \Phi^s_E (\cdot,\cdot).
\]
Given an initial condition $(\gamma(\xi),v)$, we can define the set
\begin{equation}\label{eq:T-}
	\mathbb{T}^-(\xi,v) \eqdef
	\left\lbrace 
	s_1>0\;\middle\lvert\ 
	\begin{aligned} 
		& \Phi^{s_1}_E (\gx,v) =(\gamma(\xi_1),v_1) \text{ for some } (\xi_1,v_1) \in \Xi_r^- \\ 
		& \Pi_z \Phi^{s}_E (\gx,v) \notin \overline D  \text{ for every } s \in (0,s_1)
	\end{aligned}
	\right\rbrace.
\end{equation}
which contains at most one element.
In view of the outer-arc property, there holds
\[
\left\{ (\xi,v) \in \Xi^+_r \colon \mathbb{T}^-(\xi,v)  \neq \emptyset \right\} \neq \emptyset \quad \text{for all } r \in \mathcal{I}.
\]
Let us now suppose that $\mathbb{T}^-(\xi,v) =\{s_1\}$ and call 
\[
(\gamma(\xi_1),v_1) \eqdef \Phi^{s_1}_E (\gx,v).
\]
To proceed with the inner dynamics we need to refract our arc hence we define $(\gamma(\xi_1),v'_1) \eqdef (\gamma(\xi_1),R_{EI}(v_1;\xi_1,h))$  to start with an inner arc. We consider  the flow associated to the inner problem, $\Phi^s_I (\gamma(\xi_1),v'_1)$ 
and the set
\begin{equation}\label{eq:T+}
	\mathbb{T}^+(\xi,v) \eqdef
	\left\lbrace 
	s_2>0\;\middle\lvert\ 
	\begin{aligned} 
		& \Phi^{s_2}_I (\gamma(\xi_1),v'_1) =(\gamma(\xi_2),v'_2) \text{ for some } (\xi_2,v'_2) \in \bigcup_{r' \in NA(r)}\Xi_{r'}^+ \\ 
		& v_2' \in \tilde B_I^{\xi_2} \\ 
		& \Pi_z \Phi^{s}_I (\gamma(\xi_1),v'_1) \in D \text{ for every } s \in (0,s_2)
	\end{aligned}
	\right\rbrace.
\end{equation}
Once more, $\mathbb{T}^+(\xi,v)$ has at most one element and we can extend its definition to every pair $(\xi,v) \in \Xi^+_r$ by requiring that 
\[
\mathbb{T}^-(\xi,v) = \emptyset 
\quad \implies \quad 
\mathbb{T}^+(\xi,v) = \emptyset.
\]
In view of Theorem \ref{thm:SymDyn}, there holds again that 
\[
\left\{ (\xi,v) \in \Xi^+_r \colon \mathbb{T}^+(\xi,v)  \neq \emptyset \right\} \neq \emptyset \quad \text{for all } r \in \mathcal{I};
\]
indeed, it is sufficient to consider a word $\underline \ell$ of length at least 2 with the first element equal to $r$ and take the initial condition of the corresponding trajectory. \\
Let now $(\xi,v) \in \Xi^+_r$ be such that $\mathbb{T}^+(\xi,v) \neq \emptyset$ and let $(\xi_2,v_2')$ as in Eq. \eqref{eq:T+}. As $v_2' \in \tilde B_I^{\xi_2}$ we are allowed to define 
\begin{equation} \label{eq:v2}
	v_2 \eqdef R^{-1}_{EI}(v_2';\xi_2,h)
\end{equation}
so that $(\xi_2,v_2) \in \Xi^+_{r'}$ and we have the initial condition for a second outer arc.

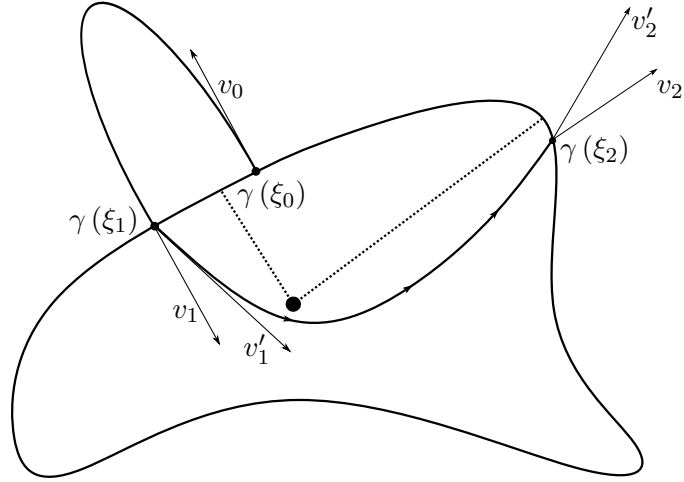
\begin{figure}
	\centering
	\begin{overpic}[width=0.5\linewidth]{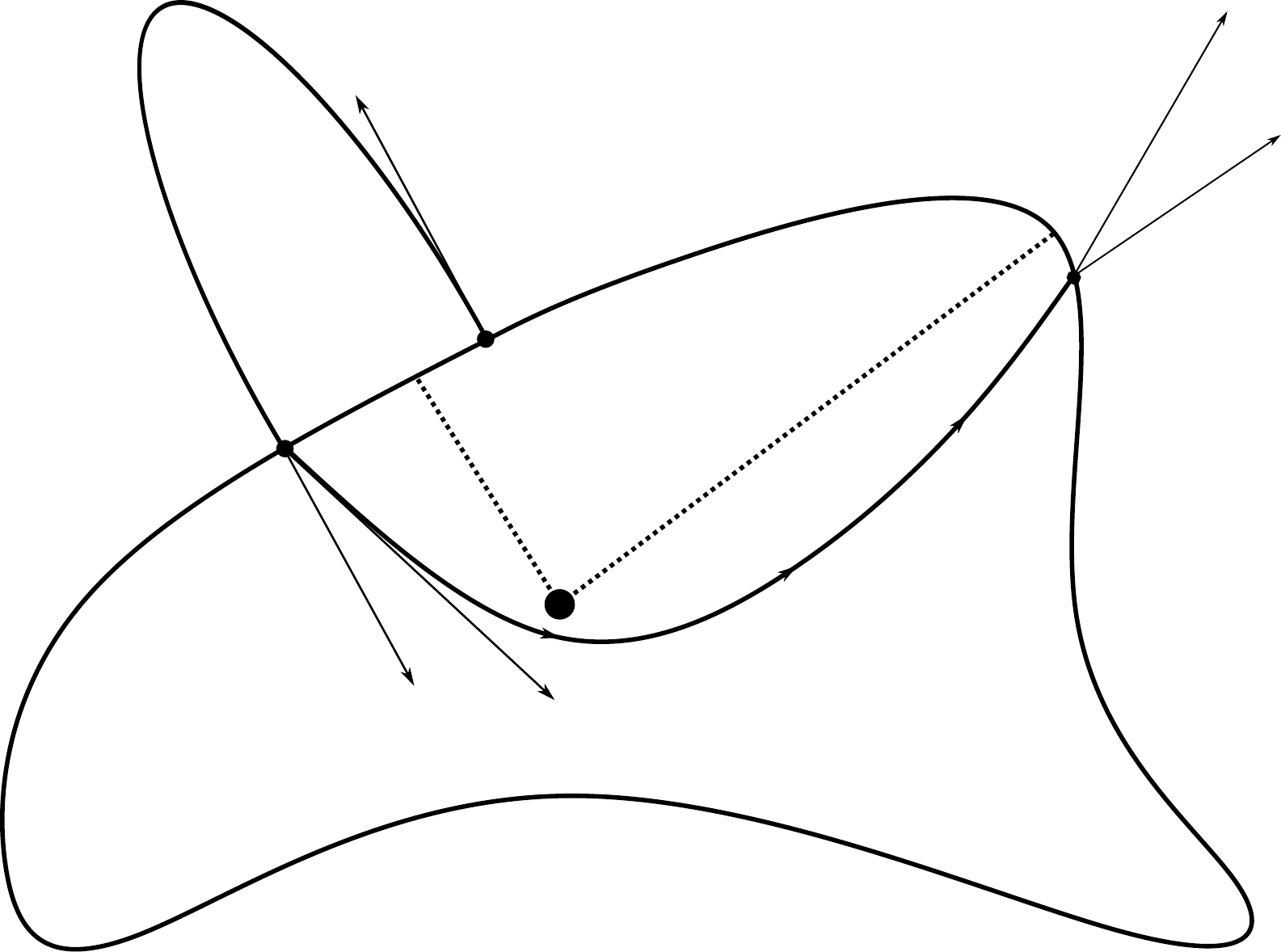}
		\put (35, 43){\small$\gamma\left(\xi_0\right)$}
		\put (32, 60){\small$v_0$}
		\put (25, 25){\small$v_1$}
		\put (36, 19){\small$v'_1$}
		\put (9, 39){\small$\gamma\left(\xi_1\right)$}
		\put (85, 50){\small$\gamma\left(\xi_2\right)$}
		\put (96, 70){\small$v'_2$}
		\put (100, 60){\small$ v_2$}
	\end{overpic}
	\caption{An example of concatenation starting from a point $(\gamma(\xi_0),v_0)$, with $(\xi_0,v_0)\in A$.}
	\label{fig:dinamica-simbolica}
\end{figure}

The following non-empty set contains pairs $(\xi,v)$ for which it is possible to construct a complete concatenation outer-inner arc and where a first return map is well defined.
\begin{definition}\label{def:setA}
	Let us define the set $A\subset[0,L]\times\R^2$ as the set of pairs $(\xi,v)$ such that
	\begin{enumerate}
		\item $(\xi,v) \in \Xi^+_{r}$ for some $r \in \mathcal{I}$;
		\item there exist $s_1 \in \mathbb{T}^-(\xi,v)$ and $s_2 \in \mathbb{T}^+(\xi,v)$;
	\end{enumerate}
	and the first return map
	\begin{equation} \label{eq:mapF}
		F \colon A \to F(A), \qquad F(\xi,v) \eqdef(\xi_2,v_2)
	\end{equation}
	where $\xi_2$ and $v_2$ are introduced respectively in \eqref{eq:T+} and \eqref{eq:v2}, see Figure \ref{fig:dinamica-simbolica}.
\end{definition}
The function $F$ is clearly a bijection and, by recurrence, we construct the set $X$ as the set on which all the positive and negative iterates of $F$ are defined. Let
\[
X^+_1 \eqdef A
\quad \text{and} \quad 
X^-_1 \eqdef A \cap F(A) =  X^+_1 \cap F(X^+_1 )
\]
and observe that $X^-_1 \neq \emptyset$ (by Theorem \ref{thm:SymDyn}); moreover both $F$ and $F^{-1}$ are well defined on $X^-_1$.
Then, for any positive integer $k \geq 2$, we introduce the non-empty sets 
\[
X^+_k \eqdef X^-_{k-1}\cap F^{-1}(X^-_{k-1})
\quad \text{and} \quad 
X^-_k \eqdef X^+_k \cap F(X^+_k)
\]
so that on $X^-_k $ the iterates $F^j$, with $j \in \Z$ and $|j| \leq k$, are well defined (in our notation $F^0 = \mathrm{Id}$). 
We are then ready to define the set of initial conditions in $\Xi^+_r$ that generate trajectories with an infinite number of transversal intersections with $\partial D$, namely
\[
X \eqdef \bigcap_{k \in \N}X^-_k.  
\]    
By virtue of Theorem \ref{thm:SymDyn}, the set $X$ is non empty, and on $X$ are defined the iterates $F^j$, for any integer $j \in \Z$.
The set $X$ is the set of initial conditions which generate trajectories for the complete dynamics that cross $\partial D$ an infinite
number of times; hence it is invariant for the first return map $F$ an then it is convenient to consider the restriction
\[
\mathcal{F} \eqdef F|_X
\]

Let us now consider the set of bi-infinite admissible words
\begin{equation}
	\mathcal{L} \eqdef 
	\left\{ \underline{\ell} \in \mathcal{I}^{\Z} \colon \ell_{j+1} \in NA(\ell_j), \,\forall j \in \Z \right\}
\end{equation}
endowed with the metric
\begin{equation}\label{eq:distL}
	d(\underline\ell,\underline{m}) \eqdef \sum_{k \in \Z} \frac{\rho(\ell_k,m_k)}{4^{|k|}},
\end{equation}
where $\rho(i,j)=0$ if $i=j$ and $\rho(i,j)=1$ if $i\neq j$;
we refer to the book \cite{MR1995704} for a complete treatment on the subject.
It is straightforward to prove that, with this metric, the subset of the periodic bi-infinite words
\[
\mathbb{L}_P  \eqdef 
\left\{ \underline{\ell} \in \mathcal{L} \colon \underline\ell \text{ is periodic} \right\}
\]  
is dense in $\mathcal{L}$. Furthermore, we observe that the elements of $\mathbb{L}_P$ are the periodic extensions of the finite admissible words of the set $\mathbb{L}$ introduced in Eq. \eqref{eq:LL}.

We now introduce the map $\chi \colon X \to \mathcal{I}$ such that 
\begin{equation}\label{eq:defchi}
	\chi(\xi,v) = r \Longleftrightarrow \xi \in {I^{(r)}}
\end{equation}
and the projection map 
\[
\pi \colon X \to \mathcal{L}, 
\quad  
(\xi,v) \mapsto \pi(\xi,v)=(\ell_j)_{j \in \Z}
\quad \text{with } \ell_j = \chi ({\mathcal F}^j(\xi,v))
\]
through which we are able to associate to every $(\xi, v)$ the bi-infinite word realized by the trajectory of initial condition $(\gx,v)$. \\
We can then consider the commutative diagram
\[
\begin{tikzcd}
	X \arrow{r}{\mathcal{F}} \arrow{d}{\pi} & X \arrow{d}{\pi} \\
	\mathcal{L} \arrow{r}{\sigma_r}	& \mathcal{L}
\end{tikzcd}
\]
where $\sigma_r$ is the Bernoulli right-shift. In order to prove Theorem \ref{thm:ex_dyn_sym_intro}, we are left to show that the map $\pi$ is a continuous surjection. 

\begin{proposition}\label{prop:pi_suriettiva}
	If $h > h_1$, the projection map $\pi$ is surjective. 
\end{proposition}
\begin{proof}
	Let us take a sequence $\underline{\ell}=(\ell_j)_{j\in\Z}\in\mathcal L$, and, for every $n\in\N$ consider the truncated sequences 
	\begin{equation*}
		\underline\ell^{(n)}\eqdef (\ell_{-n},\ldots, \ell_n), 
	\end{equation*}
	which are elements of $\mathbb L$ of length $2n+1$. Since $h> h_1$, by Theorem \ref{thm:SymDyn}, for every $n\in\N$ there exists $\hat{\underline\xi}^{(n)} \in \mathbb{S}_{\underline\ell^{(n)}}$ such that the corresponding trajectory 
	\[
	z^{(n)}(\cdot) \eqdef z\left(\cdot;\hat{\underline\xi}^{(n)},h\right)
	\]
	realizes the word $\underline\ell^{(n)}$.  Since the trajectories $z^{(n)}(\cdot)$ are periodic, without loosing in generality and possibly with a time translation, we can assume that 
	\[
	z^{(n)}(0) \in \gamma (\overline{I^{(\ell_0)}}).
	\]
	We now extend by periodicity $\hat{\underline\xi}^{(n)}$ to obtain a sequence in $\mathcal{L}$
	\[
	\left(\underline\xi^{(n)}\right)_n, \qquad 
	\underline\xi^{(n)} = \left(\xi^{(n)}_k\right)_{k \in \Z},
	\]
	where, for every $n$, $\underline\xi^{(n)} \in \mathbb{L}_P$ is a $(4n+2)$-periodic sequence. By construction,
	\begin{equation}\label{eq:comp}
		\forall k \in \Z,\; \exists N_k >0, \; \exists i_k \in \mathcal{I}
		\; \colon \; 
		\xi^{(n)}_k \in {I^{(i_k)}}, \; \forall n \geq N_k,
	\end{equation}
	namely, fixed $k$, the points $\xi_k^{(n)}$ belong eventually to the same compact interval prescribed by the sequence $\underline \ell$.  Through a diagonal process, we can construct an index sequence $(a_n)_n \subset \mathbb{N}$ with $a_n \to \infty$ such that
	\[
	\xi^{(a_n)}_i  \stackrel{n \to \infty}{\longrightarrow} \bar\xi_i, \quad \text{for every } i \in \Z,
	\]
	where the sequence $\bar{\underline{\xi}} \eqdef (\bar\xi_i)_i$  realizes the word $\underline \ell$, in the sense that
	\begin{equation*}
		\bar\xi_{2j}, \bar\xi_{2j+1} \in {I^{(\ell_j)}},
		\qquad 
		\text{for every } j \in \Z.
	\end{equation*}
	Since $h>h_1$, we can then define the concatenation $\bar z\colon \R \to \R^2$ connecting the sequence of points $\left( \gamma(\bar \xi_i) \right)_i$: our aim is now to verify that $\bar z$ is an admissible trajectory for the complete dynamics. 
	This fact follows from the differentiable dependence of each arc and recalling that the concatenations $\left(z^{(a_n)}\right)_n \subset \left(z^{(n)}\right)_n$ satisfy the Snell's law at every transition point.
	
	Defining now $(\xi,v) \eqdef (\bar \xi_0, \bar z'(0))$ we have that $\pi(\xi,v) = \underline \ell$ and the surjectivity is proved.
\end{proof}

In order to prove the continuity of $\pi$ let us start with a preliminary lemma whose proof follows from the continuous dependence of the outer and inner arcs with respect to variations of the endpoints and by the compactness of the intervals ${I^{(j)}}$, $j\in\mathcal I$. 
\begin{lemma}\label{lem:T_unif_lim}
	Let $h> h_1$; then there exists $C>0$ such that for every $j\in\mathcal I$, $\xi_0, \xi_1\in {I^{(j)}}$ and $\xi_2\in \bigcup_{i\in NA(j)}{I^{(i)}}$ it holds	
	\begin{equation*}
		T_E(\xi_0, \xi_1)\leq C, \quad T_I(\xi_0, \xi_2; h)\leq C, 
	\end{equation*}
	where $T_E$ and $T_I$ are as in Definition \ref{def:concatenazione}. 
\end{lemma}

We are now ready to verify the continuity of the map $\pi$; we recall that the space of the admissible words $\mathcal L$ is endowed with the distance $d(\cdot, \cdot)$, defined in Eq. \eqref{eq:distL}, while in $X$ we will consider the usual Euclidean metric over $\R^3$. 
\begin{proposition}\label{prop:pi_continua}
	If $h> h_1$, the projection map $\pi$ is continuous. 
\end{proposition}
\begin{proof}
	Let us fix $(\xi_0, v_0)\in X$. In order to estimate the quantity $d\left(\pi(\xi, v), \pi(\xi_0, v_0)\right)$ we define, for every $k\in\Z$, the $k-th$ projection map 
	\begin{equation*}
		\pi_k\colon X\to \mathcal I, \qquad \pi_k(\xi, v) \eqdef \chi\left(\mathcal F^k(\xi, v)\right), 
	\end{equation*}
	where $\chi$ is defined as in Eq. \eqref{eq:defchi}, to have
	\begin{equation}
		d\left(\pi(\xi, v), \pi(\xi_0, v_0)\right) = \sum_{k\in \Z}\frac{\rho\left(\pi_k(\xi, v), \pi_k(\xi_0, v_0)\right)}{4^{|k|}}.
	\end{equation}
	As the above series is always convergent, for any $\epsilon >0$ there exists $k_0\in \N$ such that for every $(\xi, v)\in X$
	\begin{equation*}
		\sum_{|k|\geq k_0}\frac{\rho\left(\pi_k(\xi, v), \pi_k(\xi_0, v_0)\right)}{4^{|k|}}<\epsilon.
	\end{equation*}
	We will now prove that, if $\delta>0$ is small enough and $\|(\xi, v)-(\xi_0, v_0)\|<\delta, $ then 
	\begin{equation}\label{eq:serietronc}
		\sum_{|k|< k_0}\frac{\rho\left(\pi_k(\xi, v), \pi_k(\xi_0, v_0)\right)}{4^{|k|}}=0.
	\end{equation}
	This is equivalent to require that, if $\underline \ell\eqdef\pi(\xi, v)$ and $\underline m\eqdef\pi(\xi_0, v_0)$, then $\ell_k=m_k$ for every $|k|<k_0$, namely, that the trajectories generated by $(\xi, v)$ and $(\xi_0,v_0)$ intersect the boundary $\partial D$ in the same neighbourhoods in the first $k_0$ steps forward and backward. In view of Lemma \ref{lem:T_unif_lim} it is possible to find $a>0$ such that both the trajectories cross $\partial D$ at least $4k_0+3$ times within the time interval $[-a, a]$. The thesis then follows from the continuous dependence on the initial data on  $[-a, a]$.  
\end{proof}

We conclude this section with a final remark; following exactly the techniques used in Proposition \ref{prop:pi_suriettiva} we can prove a result which will be crucial in Section \ref{sec:non_int_chaos}. 

\begin{proposition}\label{prop:mezze_etero}
	Let 
	\begin{equation*}
		\mathcal{L}^+\eqdef\left\{(\ell_0,\ldots, \ell_n, \ldots)\in\mathcal I^{\N}\,\colon\,\ell_{i+1} \in NA(\ell_i)\;\forall i\in \N\right\}.
	\end{equation*}
	be the set of the admissible \emph{infinite} words,
	$\underline\ell\in\mathcal L^+$ and $\xi\in I^{(\ell_0)}$. The there exists $v\in\R^2$ such that $(\xi, v)\in X$ and $\pi(\xi,v)=\underline\ell$\footnote{Here, the projection is intended only for the forward trajectory starting from $\left(\gamma(\xi), v\right)$.}. 
\end{proposition}

\subsection{Non-collisional symbolic dynamics}\label{ssec:collisioni}
As observed in Definition \ref{def:prop_arco_int},
the inner dynamics naturally includes collisional arcs as well. As a consequence, we can not \emph{a priori} exclude the occurrence of collisions in the symbolic dynamics found by means of Theorem \ref{thm:ex_dyn_sym_intro}. Nevertheless, as we will prove in the present section, if we suitably restrict the set $\mathcal L$ we can construct a \emph{non-collisional} symbolic dynamics.

Let $z$ be a concatenation satisfying $\underline{\ell} \in \mathcal{L}$ and assume that $z$ has a collision, which means that one of its inner arcs is homothetic. Assume, without loosing in generality that such collisional inner arc hits the boundary in $\gamma\left(I^{(\ell_0)}\right)$. By uniqueness of the solution of a Cauchy problem and by Snell's law we deduce that after the collision the trajectory reflects into itself. This reflection, which involves every arc in the concatenation, forces the symmetry of the word $\underline{\ell}$ with respect to $\ell_0$, in the sense that $\ell_k = \ell_{-k}$ for any $k \in {\mathbb{N}}$. \\
On the other hand, the same reasoning can be used in a \emph{reversed} formulation to have \emph{sufficient} conditions for the concatenation induced by $\underline\ell$ to be non-collisional. In particular one can construct a \emph{non-collisional symbolic dynamics} for our refractive billiard, as stated in Corollary \ref{coro:no_coll_intro}. 

If the collisional trajectory is periodic a more complex phenomenology appears. Let $\underline\ell= \left(\ell_0, \dots, \ell_{n-1}\right)\in\mathbb L$, 
and  $z\left(\cdot; \hat{\underline\xi}; h\right)$ the trajectory realizing $\underline\ell$ where $\hat{\underline\xi}\in\mathring{\mathbb  S}_{\underline\ell}$ is provided in Thereom \ref{thm:SymDyn} (once more we assume $h>h_1$). From this moment on, we will identify $\underline \ell$ with any of its shifts.

To describe in details the trajectory $z\left(\cdot; \hat{\underline\xi}; h\right)$, it is worth to keep trace not only of the transition points $\gamma\left(\hat{\underline\xi}_i\right)$, $i=0, \dots, 2n-1$, but also of the inner and outer velocity vectors at such points (see also Figure \ref{fig:collisioni1}, left)\footnote{From this moment on, for the sake of brevity and with an abuse of notation, we will denote with $T$ the final time of an arc, omitting the dependence on the endpoints and without discerning between the inner and the outer case.}, 
\begin{equation}\label{eq:defV}
\footnotesize{	
	\begin{aligned}
		&v_{2j}=z'_E\left(0; \gamma\left(\hat{\xi}_{2j}\right), \gamma\left(\hat{\xi}_{2j+1}\right)\right), \, &&v_{2j+1}=z'_E\left(T; \gamma\left(\hat{\xi}_{2j}\right), \gamma\left(\hat{\xi}_{2j+1}\right)\right), \, &&&v_{2n}=v_0,\\
		&\tilde v_{2j+1}=z'_I\left(0; \gamma\left(\hat{\xi}_{2j+1}\right), \gamma\left(\hat{\xi}_{2j+2}\right); h\right), \, &&\tilde v_{2j+2}=z'_I\left(T; \gamma\left(\hat{\xi}_{2j+1}\right), \gamma\left(\hat{\xi}_{2j+2}\right); h\right), \, &&&\tilde v_0=\tilde v_{2n},
	\end{aligned}
}
\end{equation}
where  $j=0, \dots, n-1$.

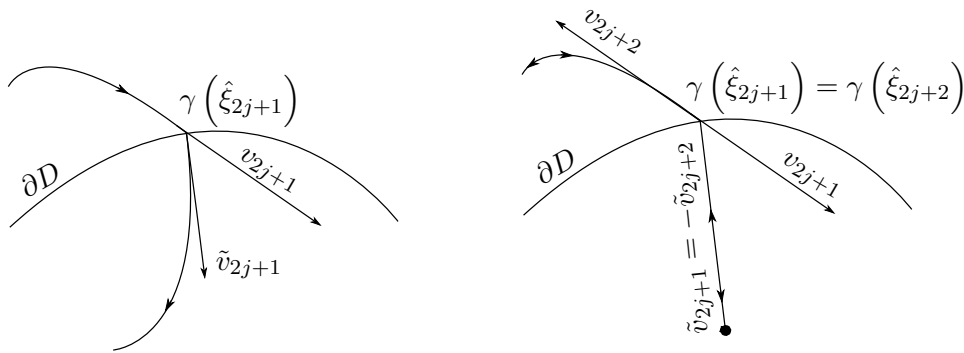
\begin{figure}
	\centering
	\begin{overpic}[width=0.7\linewidth]{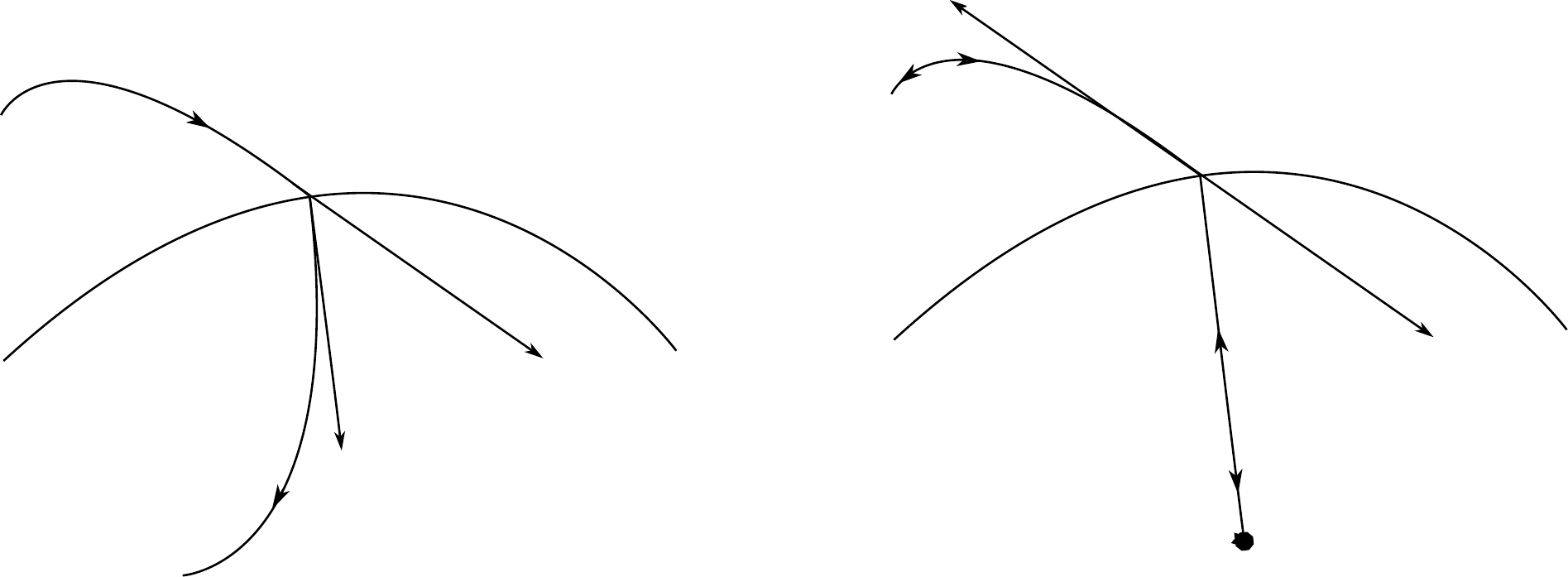}
		\put (1,17){\rotatebox{30}{\small$\partial D$}}
		\put (19,27){\rotatebox{0}{\small$\gamma\left(\hat\xi_{2j+1}\right)$}}
		\put (25,21.5){\rotatebox{-30}{\small$v_{2j+1}$}}
		\put (23,9){\rotatebox{0}{\small$\tilde v_{2j+1}$}}
		\put (58,18.5){\rotatebox{30}{\small$\partial D$}}
		\put (75,28.5){\rotatebox{0}{\small$\gamma\left(\hat\xi_{2j+1}\right)=\gamma\left(\hat\xi_{2j+2}\right)$}}
		\put (63,37){\rotatebox{-30}{\small$v_{2j+2}$}}
		\put (85,21){\rotatebox{-30}{\small$v_{2j+1}$}}
		\put (73,2){\rotatebox{96}{\small$\tilde v_{2j+1}=-\tilde v_{2j+2}$}}
	\end{overpic}
	\caption{Velocity vectors as defined in Eq. \eqref{eq:defV} in the non-collisional case (left) and in the collisional case (right). In the second case, the collisional inner arc forces a reflection in the adjacent outer arcs. }
	\label{fig:collisioni1}
\end{figure}

The following Theorem underlines as the presence of collisional inner arcs can impact the overall structure of $z\left(\cdot; \hat{\underline\xi}; h\right)$, when it is periodic. 

\begin{theorem}\label{thm:orbite di collisione}
	Let $\underline\ell\in\mathbb L$ and $h > h_1$; define $n\eqdef|\underline\ell|$, and suppose $n>1$. Let $\hat{\underline\xi}\in\mathbb{S}^\circ_{\underline{\ell}}$ and $z(\cdot)\eqdef z\left(\cdot; \hat{\underline\xi}; h\right)$ as in Theorem \ref{thm:SymDyn}, and suppose that the concatenation $z(\cdot)$ admits a collisional inner arc. Then: 
	\begin{enumerate}
		\item if $n$ is even, $z(\cdot)$ has another collisional arc and ${\underline\ell}$ is symmetric with respect to the axis that separates ${\ell}_{n/2-1}$ from $\ell_{n/2}$; 
		\item if $n$ is odd, then $z(\cdot)$ has a homothetic outer arc and $\left(\ell_0, \dots, \ell_{n-1}, \ell_0\right)$ is symmetric with respect to the axis that separates ${\ell}_{(n+1)/2-1}$ from $\ell_{(n+1)/2}$. 
	\end{enumerate}
\end{theorem}

	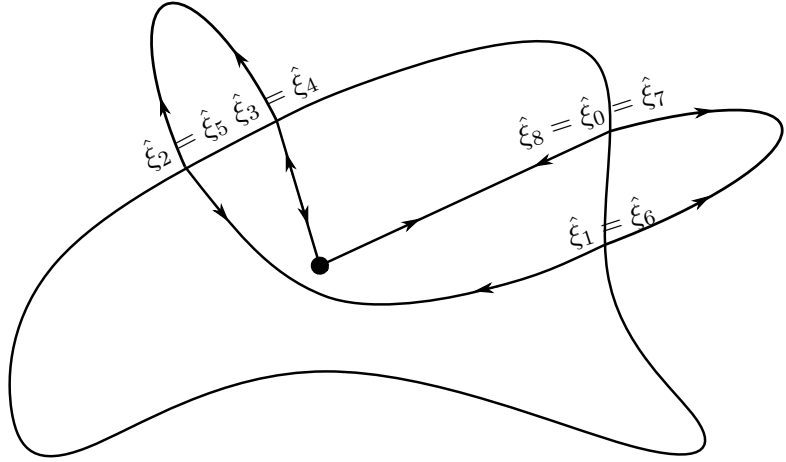
\begin{figure}
		\centering
		\begin{overpic}[width=0.6\linewidth]{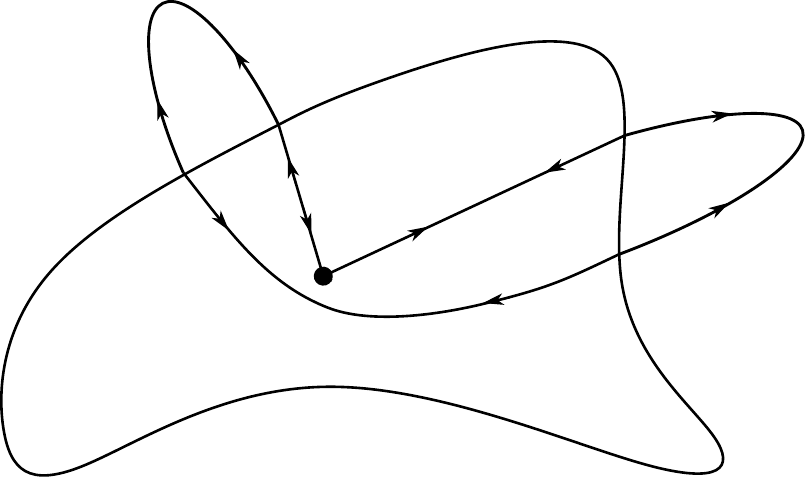}
			\put(16.5, 37.5){\small\rotatebox{28}{$\hat\xi_2=\hat\xi_5$}}	
			\put(28, 43.5){\small\rotatebox{25}{$\hat\xi_3=\hat\xi_4$}}	
			\put(65, 40.5){\small\rotatebox{18}{$\hat\xi_8=\hat\xi_0=\hat\xi_7$}}
			\put(71.5, 27.5){\small\rotatebox{18}{$\hat\xi_1=\hat\xi_6$}}		
		\end{overpic}
		\caption{Graphical representation of a possible periodic collisional trajectory. The presence of a collisional inner arc, along with the periodicity of the whole concatenation and the total number of transition points, implies the existence of a second collisional arc. With an abuse of notation $\hat \xi_i$ is identified with $\gamma \left(\hat \xi_i\right)$.}
		\label{fig:collisioni-2}
	\end{figure}

\begin{proof}[Proof of Theorem \ref{thm:orbite di collisione}]
	Let us start by assuming that $n$ is even:  without loss of generality we can suppose that the collisional arc is the one connecting $\gamma\left(\hat{\xi}_{n-1}\right)$ to $\gamma\left(\hat{\xi}_{n}\right)$, namely, $z_I\left(\cdot; \gamma\left(\hat{\xi}_{n-1}\right), \gamma\left(\hat{\xi}_n\right); h\right)$. Then, the trajectory is reflected back after the collision and one has the following equalities:
	\begin{equation}\label{eq:uguaglianzeColl}
		\footnotesize{
		\begin{aligned}
			&\hat{\xi}_{n-(j+1)}=\hat{\xi}_{n+j}, \quad\tilde v_{n-(j+1)}=-\tilde v_{n+j}, \quad v_{n-(j+1)}=-v_{n+j}, 
			\quad &&j \in \{0, \dots, n-1\}\\
			&z_E\left(\cdot; \gamma\left(\hat{\xi}_{n-(j+2)}\right), \gamma\left(\hat{\xi}_{n-(j+1)}\right)\right) = z_E\left(T-\cdot\,; \gamma\left(\hat{\xi}_{n+j}\right), \gamma\left(\hat{\xi}_{n+j+1}\right)\right),	 &&j \in \{0, \dots, n-2\} \text{ even},\\ 
			&z_I\left(\cdot; \gamma\left(\hat{\xi}_{n-(j+2)}\right), \gamma\left(\hat{\xi}_{n-(j+1)}\right); h\right) = z_I\left(T-\cdot\,; \gamma\left(\hat{\xi}_{n+j}\right), \gamma\left(\hat{\xi}_{n+j+1}\right); h\right),	 &&j \in \{0, \dots, n-2\} \text{ odd}.
		\end{aligned}
	}
	\end{equation}
	In particular, taking $j=n-1$, one has $\hat{\xi}_0=\hat\xi_{2n-1}$; by periodicity, $\hat{\xi}_{2n}=\hat{\xi}_{2n-1}$. Hence, by uniqueness, since it has the same endpoints, the inner arc $z_I\left(\cdot; \gamma\left(\hat{\xi}_{2n-1}\right), \gamma\left(\hat{\xi}_{2n}\right), h\right)$ must be collisional, and then the first claim is proved.\\ 
	Let us now focus on the structure of $\underline{\ell}=\left(\ell_0,\dots, \ell_{n-1}\right)$. Since for every $k=0, \dots, n-1$ one has $\hat{\xi}_{2k}, \hat{\xi}_{2k+1}\in I^{(\ell_k)}$, equalities in \eqref{eq:uguaglianzeColl} imply that for every $j=0, \dots, n/2-1$ it holds $\ell_{n/2-(j+1)}=\ell_{n/2+j}$, and then the sequence $\left(\ell_0,\dots, \ell_{n-1}\right)$ is symmetric with respect to the axis which separates $\ell_{n/2-1}$ from $\ell_{n/2}$. \\
	Let us now suppose that $n>1$ is odd and, without loss of generality, assume that the arc $z_I\left(\cdot; \gamma\left(\hat{\xi}_{n}\right), \gamma\left(\hat{\xi}_{n+1}\right), h\right)$ is collisional. Using the same reasoning as in the even case, one can conclude that $z_E\left(\cdot; \gamma\left(\hat{\xi}_{0}\right), \gamma\left(\hat{\xi}_{1}\right)\right)$ is homothetic and that for every $j=0, \dots, (n+1)/2-2$ one has $\ell_{(n+1)/2-(j+1)}=\ell_{(n+1)/2+j}$.
\end{proof}

\begin{remarks} \label{rem:coll_descr}
	The proof of Theorem \ref{thm:orbite di collisione} is particularly rich of further informations: 
	\begin{enumerate}
		\item in both the described cases one can not have more than two radial (inner or outer) arcs; in particular, the concatenation segment between two radial arcs must contain $n-1$ non-radial intermediate arcs. This is necessary to have $|\ell|=n$; 
		\item with completely analogous reasonings, one can prove that, if $z\left(\cdot, \hat{\underline\xi}, h\right)$ admits an outer homothetic arc, then 
		\begin{itemize}
			\item if $|\underline\ell|$ is even, there must be another outer homothetic arc;
			\item if $|\underline\ell|$ is odd, there must be a collision-ejection inner arc.
		\end{itemize}
		In both cases, the sequence $\underline\ell$ must satisfy symmetry properties analogous to the ones described in Theorem \ref{thm:orbite di collisione}; 
		\item  as a consequence, a periodic trajectory can have either zero or two radial arcs, between which it is reflected. 
	\end{enumerate}
\end{remarks}
\begin{remark}
	In the particular case $n=1$, one has that the concatenation $z\left(\cdot, \hat{\underline\xi}, h\right)$ is collisional if an only if it is a homothetic orbit for the complete dynamics. 
\end{remark}

\section{Admissible domains and central configurations}\label{sec:adm_dom_cc}
The result stated in Theorem \ref{thm:SymDyn} guarantees the presence of a symbolic dynamics under quite abstract assumptions on the domain $D$, namely the outer and inner-arc properties as well as the change-sign one. This section is devoted to translate these assumptions into more concrete ones. At the end we will find a general set of domains which enjoys the admissibility property. Such conditions will be strictly related to the notion of \emph{central configurations} for the domain that we now define.

\begin{definition}\label{def:cc}
	We say that $\bx\in[0,L]$ is a \emph{central configuration} for $D$ if it is a critical point of $\|\gx\|$, $\xi \in [0,L]$ and if $D$ satisfies $\lsc{\bx}$ (see Definition \ref{def:local_star_conv}).
\end{definition}

Generalizing Eq. \eqref{eq:cos} we immediately deduce that 
$\bx$ is a critical point for $\|\gx\|$ if and only if $\gamma(\bx)$ is orthogonal to $\partial D$. Furthermore, as the star convexity is assumed at a central configuration, by the regularity of $\gamma$, we deduce the existence of an open interval $(a,b)\ni\bx$ such that $D$ satisfies $\lsc{(a,b)}$. Possibly reducing this interval, the outer-arc property is guaranteed.

\begin{proposition}\label{prop: esistenza esterni}
	Let $\bx\in[0,L]$ be a central configuration for $D$. Then there exists a neighborhood $\bx \in (\bar a,\bar b) \subset [0,L]$ which satisfies $\lsc{(\bar a,\bar b)}$ and the outer-arc property.
\end{proposition}
The proof is based on the regularity of $\gamma$, and on the differentiable dependence of the solution of a Cauchy problem on its initial data. In particular one can start from the existence of the homothetic outer solution at $\gamma(\bx)$ (for more details see \cite[Theorem 3.1]{deblasiterraciniellissi}). In particular, if $b-a$ is small enough, we can ensure that all the outer arcs are not tangent to the boundary.
Let us outline that the outer dynamics is \textit{local} in the sense that an outer arc connects two points on $\partial D$ belonging to a neighborhood of the same point $\gamma(\bx)$. 

Let us now pass to the inner dynamics. as already seen in Section \ref{sec:symb_dyn}, they act as \emph{transfer trajectories} between different regions od $\partial D$. It is then necessary to consider not only single central configurations but a set of the latters, having suitable mutual properties. 

\begin{proposition}\label{prop: esistenza interni}
	Let $\bx_1, \ldots, \bx_m$ be central configurations, $m\geq2$, which are not antipodal  if $m=2$. Then there exists $\mathcal A=\bigcup_{i\in \mathcal I}(a_i, b_i)$ that satisfies the inner-arc property and such that 
	$\bx_i\in(a_i,b_i)$, for every $i \in \mathcal I$.
\end{proposition} 

\begin{proof}
	Let us start by observing that if $m>3$ any central configuration $\bx_i$ admits at least a non-antipodal (different for $\bx_i$ itself) one. Then there exists a disjoint union of open intervals $\tilde{\mathcal A}=\bigcup_{i=1}^m(\tilde a_i, \tilde b_i)$ satisfying \textit{(IP1)} and such that $\partial D$ is $\lsc{(\tilde a_i. \tilde b_i)}$ for every $i$. Let us now take $\mathcal A=\bigcup_{i=1}^m(a_i, b_i)$ such that for every $i$ it holds $\tilde a_i<a_i<\bx_i<b_i<\tilde b_i$: we want to show that it satisfies \textit{(IP2)}. \\
	For every $h>0$, the existence and uniqueness of a \tnt arc connecting two non-antipodal points $\gamma(\xi_1)$ and $\gamma(\xi_2)$, $\xi_1,\xi_2\in\mathcal A,$ with energy $\mathcal E+h$, called $z_I(\cdot)\eqdef z_I(\cdot;\gamma(\xi_1), \gamma(\xi_2); \mathcal E+h )$, is guaranteed by Lemma \ref{lem:lemma_sus}. We have then to ensure that all these arcs are completely contained in the domain $D$. 
	
	 Defining, for every $\xi_1,\xi_2\in\mathcal A,$ the broken line  $c(\gamma(\xi_1)\,0\,\gamma(\xi_2))$ as the union of the two straight-line segments from $\gamma(\xi_1)$ to $0$ and from $0$ to $\gamma(\xi_2)$, from \cite[p. 274]{battin} one has that 
	\begin{equation}\label{eq: convergenza arco interno}
		\lim_{h\to\infty}dist\left(z_I\left([0, T]\right), c(\gamma(\xi_1)\,0\,\gamma(\xi_2))\right)=0,  
	\end{equation}
$T>0$ being such that $z_I(T)=\gamma(\xi_2)$. \\
\begin{figure}
	\centering
	\begin{overpic}[width=0.5\linewidth]{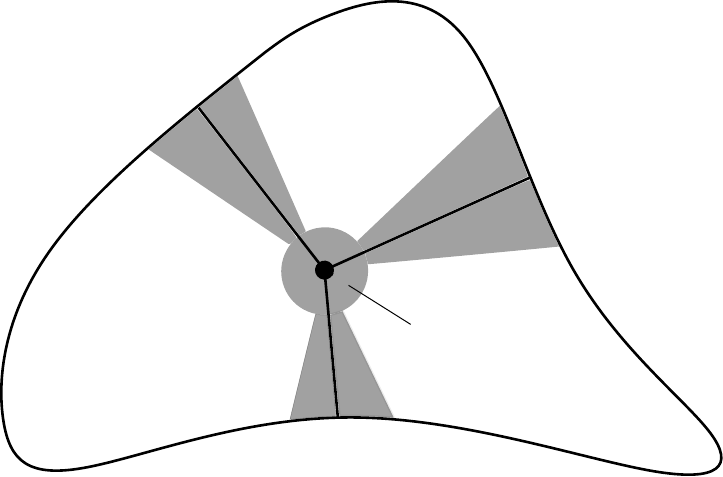}
		\put (17, 50) {\rotatebox{35}{$\gamma\left(\bx_2\right)$}}
		\put (74.5, 41) {$\gamma\left(\bx_1\right)$}
		\put (71, 50) {$\gamma\left(\tilde a_1\right)$}
		\put (78, 32) {$\gamma\left(\tilde b_1\right)$}
		\put (43, 1) {$\gamma\left(\bx_3\right)$}
		\put (26, 42.5) {$T_2$}
		\put (63, 40) {$T_1$}
		\put (47.5, 10) {$T_3$}
		\put (57, 19) {$B_\rho(0)$}
	\end{overpic}
	\caption{The open set $A$ introduced in the proof of Proposition \ref{prop: esistenza interni}. $A$ is the union of the circular neighborhood of the origin $B_\rho(0)$ and of the sectors $T_i$, $i=1, 2,3$.}
	\label{fig:prova archi interni}
\end{figure}
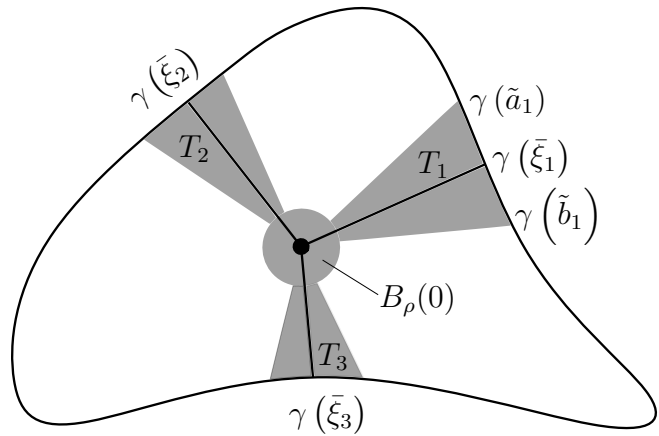
Let $\rho>0$ be such that  $B_\rho(0)\subsetneq D$ and, for every $i=1, \dots, m$, define the conic set
(see Figure \ref{fig:prova archi interni})
\begin{equation*}
	T_i=\left\{\lambda \gamma(\xi)\text{ }|\text{ }\lambda\in(0,1),\, \xi\in\left(\tilde a_i, \tilde b_i\right)\right\}, 
\end{equation*} 
and the open set
\begin{equation*}
	A = \bigcup_{i=1, \dots, N}T_i \cup B_\rho(0).
\end{equation*}
By the convergence in \eqref{eq: convergenza arco interno} and the finiteness of $m$, we can then ensure the existence of a threshold value for $h$ that ensures that $\mathcal A$ satisfied \emph{(IP2)}. 
\end{proof}
We stress that, possibly reducing the size of its intervals, we can find $\mathcal A$ satisfying both the inner and the outer-arc property. 

\begin{remark}\label{rem:hom}
	Propositions \ref{prop: esistenza esterni} and \ref{prop: esistenza interni} imply in particular the existence of radial solutions (collisional in the inner case) separately for the inner and outer dynamics in correspondence of every direction.
	Since at a central configuration the radial direction is orthogonal to the boundary $\partial D$, the corresponding solution is not deflected by Snell's law \eqref{eq:Snell1}. Hence, for any central configuration, our complete dynamics admits an homothetic ejection-collision solution in the direction $\gamma(\bx)$. 	
\end{remark}

\begin{remark}
	Propositions \ref{prop: esistenza esterni} and \ref{prop: esistenza interni} ensure the existence of the outer and inner dynamics separately, provided that the endpoints' parameters $\xi_1$ and $\xi_2$ are in suitable neighborhoods of $\bx_i$, $i=1, \dots, m$. Nevertheless, this is not sufficient to ensure 
	the good definition of the \emph{complete} dynamics, that is, a concatenation between outer and inner arcs satisfying the Snell's law.\\
	In particular, one has that the refraction exterior-interior is always possible, while 
	the converse, interior-exterior, can take place if and only if the inner arc is sufficiently transverse to the boundary. Hence, in order to prove the existence of a complete dynamics, we should find conditions to have uniform transversality properties of the inner arcs. On the other hand, this is not really necessary to our purposes: indeed, in Section \ref{sec:symb_dyn}, we have proved \emph{a posteriori} that the particular concatenations of outer and inner arcs that realize a symbolic dynamics for our problem
	are admissible trajectories for the complete dynamics. The transversality of an inner arc was in fact indirectly deduced from the validity of the variational formulation of Snell's law, along with the transversality of the subsequent outer arc.
\end{remark}

To link the presence of central configurations to the admissibility of the domain $D$ (see Definition \ref{def:D_admiss_intro}), we are left to guarantee the change-sign property. This will be true under some more assumptions on the central configurations themselves: we will say that a central configuration is \emph{strict} if it is a strict maximum or minimum for $\|\gamma(\cdot)\|$.
The following result is a direct consequence of the notion of strict extremal points for $C^1$ functions along with Proposition \ref{prop:cs_geom}.

\begin{proposition}\label{prop:cs_cc}
	Let $\bx_1, \ldots, \bx_m$ be strict central configurations, $m\geq2$, which are not antipodal  if $m=2$. Then there exists $\mathcal A=\bigcup_{i\in \mathcal I}(a_i, b_i)$ that satisfies the change-sign property.
\end{proposition}

\begin{remark}\label{rem:dominant2}
Let us fix $i \in \mathcal{I}$ and suppose that $\bx_i$ is a strict minimum for $\|\gamma(\cdot)\|$; let now $[\alpha_i,\beta_i]$ be the compact neighborhood provided by the change-sign property.
Recalling Lemma \ref{lem:lemma_sus}	and the geometric interpretation of the derivatives of the Jacobi distances given in Eq. \eqref{eq: derivate S} one can deduce that, for $h$ large enough, every inner arc starting from (or arriving to) $\gamma(\alpha_i)$ form with $\dot\gamma(\alpha_i)$ an angle strictly greater than $\pi/2$; the inequality is reversed if we consider $\gamma(\beta_i)$ (see Figure \ref{fig:stime-sugli-angoli}, left).
The situation is symmetric if $\bx_i$ is a strict maximum (see Figure \ref{fig:stime-sugli-angoli}, right). \\
In view of Remark \ref{rem:dominant_part}, is exactly this asymptotic behavior that let us pass from the Euclidean change-sign property to the change-sign property itself.
\end{remark}

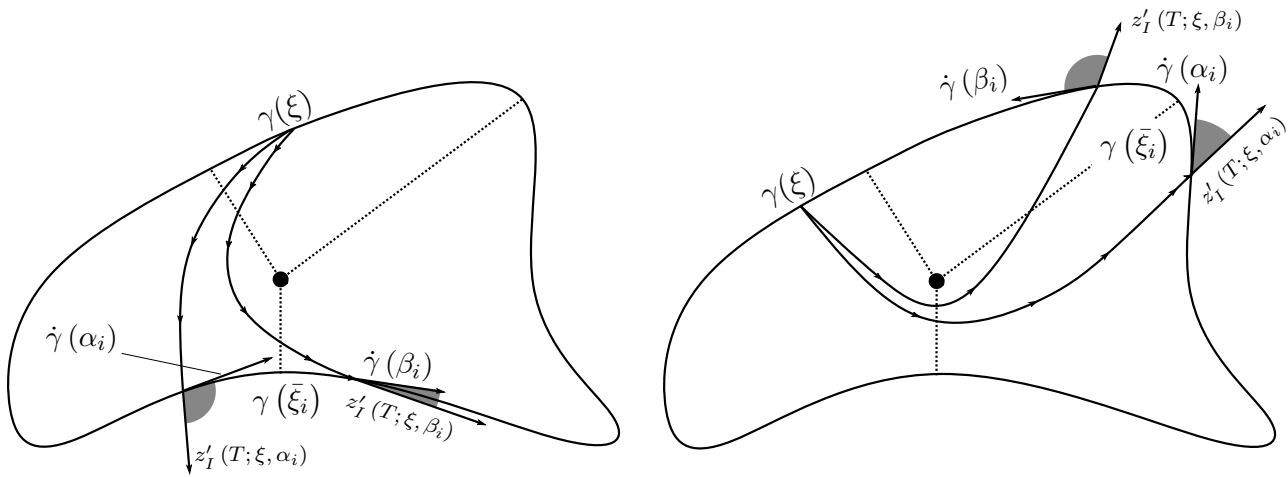
\begin{figure}
	\centering
	\begin{overpic}[width=1\linewidth]{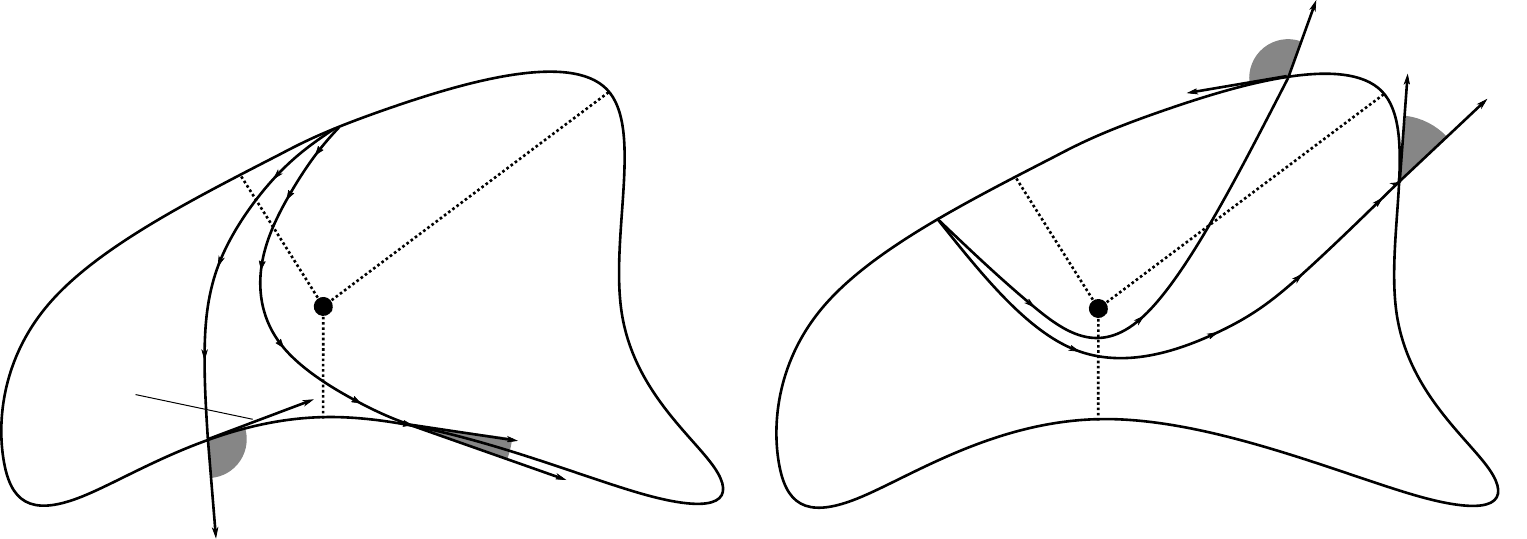}
		\put (19, 27) {\rotatebox{20}{$\gamma(\xi)$}}
		\put (19, 5.1) {\rotatebox{0}{\small$\gamma\left(\bx_i\right)$}}
		\put (3, 10) {\rotatebox{0}{\small$\dot\gamma\left(\alpha_i\right)$}}
		\put (14.5, 1) {\rotatebox{0}{\tiny$z'_I\left(T; \xi, \alpha_i\right)$}}
		\put (26, 5.3) {\rotatebox{-17}{\tiny$z'_I\left(T; \xi, \beta_i \right)$}}
		\put (27.5, 8.2) {\rotatebox{-5}{\small$\dot\gamma\left(\beta_i\right)$}}
		\put (58, 21) {\rotatebox{20}{$\gamma(\xi)$}}
		\put (87, 35) {\rotatebox{0}{\tiny$z'_I\left(T; \xi, \beta_i\right)$}}
		\put (92, 21) {\rotatebox{46}{\tiny$z'_I\left(T; \xi, \alpha_i\right)$}}
		\put (72, 29.5) {\rotatebox{10}{\small$\dot\gamma\left(\beta_i\right)$}}
		\put (89, 31) {\rotatebox{0}{\small$\dot\gamma\left(\alpha_i\right)$}}
		\put (84, 25) {\rotatebox{0}{\colorbox{white!10}{\strut\small$\gamma\left(\bx_i\right)$}}}
	\end{overpic}
	\caption{Asymptotic behavior of inner arcs, for large inner energies, in neighborhoods of strict central configurations (see Remark \ref{rem:dominant2}).}
	\label{fig:stime-sugli-angoli}
\end{figure}

Given the above construction, it turns out that the domain $D$ is admissible and Theorem \ref{thm:sym_dyn_cc} follows.

\section{Non integrability and chaos}\label{sec:non_int_chaos}
%
Until now we focused on the construction of a symbolic dynamics, which is a substantial intermediate step to the presence of chaos. In this section we fill the gap between the two concepts providing also some statements on the non integrability of the model.

{Let us start with some considerations related to the presence of infinitely many heteroclinic connections between different homothetic ejection-collision trajectories (on this topic see for example \cite{BalGilGua2022,GuaMarSea2016,GuaParSeaVid2021,MR1765636}). We will present results which hold under some further assumptions on central configurations, according to the following definition. 
	\begin{definition}
		\label{def:admiss_non_deg}	
		A central configuration $\bx$ (according to Definition \ref{def:cc}) is termed \emph{non degenerate} if $\gamma$ is of class $C^2$ in a neighborhood of $\bar \xi$ and 
		\[
		\frac{d^2}{d\xi^2}\|\gamma(\xi)\|_{|_{\xi=\bx}} \neq 0.
		\]
	\end{definition}
	
	\begin{proposition}\label{prop:selle}
		Suppose $D$ admits a non-degenerate central configuration $\bx$. 
		Then if $h$ is large enough, the homothetic trajectory in the direction of $\gamma\left(\bx\right)$ is a hyperbolic saddle equilibrium. 
	\end{proposition}
	\begin{proof}
		The proof of this result relies on asymptotic estimates based on \cite[Remark 5.1]{deblasiterraciniellissi}. Here, in particular, a general domain is considered, and, after the construction of a suitable first return map $f$, the stability of the homothetic trajectories is investigated in relation to the local geometric features of the boundary and the value of the physical parameters $\mathcal E, h, \omega, \mu$. The inspection is carried on by considering the sign of the discriminant $\Delta$ of the  characteristic polynomial associated to the Jacobian matrix of $f$ centered in an homothetic direction. By straightforward estimates, one can prove that, as long as the considered homothetic direction is nondegenerate, $\lim_{h\to\infty}\Delta>0$, and then $\bx$ is a saddle point. 
	\end{proof}
	 
Let us now assume that we are in the setting of Theorem \ref{thm:sym_dyn_cc} and let $[\alpha_i,\beta_i]$, $i \in \mathcal I$, be the compact neighborhood around each central configuration used to construct the symbolic dynamics.
The next result is a straightforward consequence of Propositions \ref{prop:mezze_etero} and \ref{prop:selle}.
	
\begin{corollary}\label{cor:mezze_etero}
	Let us suppose that there exist $\bx_1, \dots, \bx_m$ strict central configurations, not antipodal if $m=2$ and that $\bx_1$ is non-degenerate.
	Then, if $h$ is large enough, for every $\xi \in \bigcup_i[\alpha_i,\beta_i]$ there exist infinitely many half-heteroclinic connections tending forward (resp. backward) to the homothetic trajectory in the direction of $\gamma\left(\bx_1\right)$.
\end{corollary}
	
As far as more than one central configuration is non degenerate, Proposition \ref{prop:mezze_etero} allows to construct heteroclinic connections between different saddle points.

\begin{corollary}\label{cor:hetero}
	Let us suppose that there exist $\bx_1, \dots, \bx_m$ strict central configurations, not antipodal if $m=2$.
	Then, if the energy jump $h$ is large enough, for every pair of non-degenerate central configurations $\bx_i$, $\bx_j$, $i\neq j$, there exist infinitely many heteroclinic connections between the homothetic trajectories in the direction of $\gamma\left(\bx_i\right)$ and $\gamma\left(\bx_j\right)$.
\end{corollary}
\begin{proof}
	Let us suppose that $h$ is large enough such that both Theorem \ref{thm:sym_dyn_cc} and Proposition \ref{prop:selle} hold. Now, call $\bar z_i$ and $\bar z_j$ the corresponding homothetic trajectories in the direction of  $\gamma\left(\bx_i\right)$ and $\gamma\left(\bx_j\right)$, and consider the bi-infinite (non-periodic) word $\underline \ell=(\dots, i, i, i,[\ldots], j, j, j, \dots)$, where $[\ldots]$ denotes any word of finite length such that $\underline{\ell} \in \mathcal{L}$. In view of Proposition \ref{prop:pi_suriettiva} we define the sequence $(\hat \xi_k)_{k\in\mathbb Z}$ which realizes $\underline \ell$. By the hyperbolicity of both $\bar z_i$ and $\bar z_j$ as equilibrium trajectories, this sequence must belong to the unstable manifold of $\bar z_i$ as well as to the stable manifold of $\bar z_j$. 
\end{proof}}

We now connect the presence of infinitely many half-heteroclinics  to the analytic non integrability of our dynamical systems.  This result is obtained by adapting a classical argument by Kozlov (\cite{Koz1983}) and makes use of Proposition \ref{prop:mezze_etero} and Corollary \ref{cor:mezze_etero}.

\begin{theorem}\label{thm:no_an_int}
	Let us suppose that there exist $\bx_1, \dots, \bx_m$ strict central configurations, not antipodal if $m=2$, and assume that $\bx_1$ is non-degenerate. 
	Then, if $h$ is large enough, there are no analytic first integrals associated to the dynamics which are not constant.
\end{theorem}

\begin{proof}
	On every initial condition $(\xi,v) \in X$, the first return map, its inverse and all their iterates are well defined. In particular, let us observe that, fixed $\xi\in [0, L]$, every outward-pointing velocity vector $v$ starting from $\gamma(\xi)$ is uniquely determined by the angle $\alpha\in(-\pi/2,\pi/2)$ between $v$ and the normal unit vector to $\gamma$ at $\gamma(\xi)$. Using $\alpha$ as a new variable, the initial conditions  corresponding to the homothetic arcs will be denoted with $(\bx_j, 0)$, $j=1,\ldots,m$. By construction, and with a slight abuse of notation, there exists $\alpha_0>0$ such that
	\[
	X \subset \mathcal{U} \eqdef \bigcup_{i=1}^m [\alpha_i,\beta_i]\times (-\alpha_0,\alpha_0).
	\]  
	Let now $G:\mathcal{O} \to \mathbb{R}$, with $\mathcal{O}$ an open set containing $\mathcal{U}$, be an analytic first integral and let $c \in \mathbb{R}$ be such that $G(\bx_1,0)=c$.
	If $h$ is large enough, by Proposition \ref{prop:selle}, the stable and unstable manifolds of $(\bx_i,0)$ are contained in the same level set $\{G=c\}$. \\
	Fix now $\hat \xi \in \bigcup_{i=1}^m [\alpha_i,\beta_i]$: Proposition \ref{prop:mezze_etero} ensures that there exist infinite pairs $(\hat \xi,\alpha) \in \{\hat \xi\} \times (-\alpha_0,\alpha_0)$ that belong to the stable manifold of $(\bx_1,0)$. 
	This means that the $c$-level of the analytic function $G(\hat\xi,\cdot):(-\alpha_0,\alpha_0) \to \mathbb R$ admits an accumulation point. Hence this function is constant. We conclude the proof by the arbitrarity of $\hat \xi$.  
\end{proof}

We are now ready to state the results which prove the chaoticity of our model.
	\begin{lemma}\label{lem:pi_iniettiva}
		Let us suppose that there exist $\bx_1, \dots, \bx_m$ non-degenerate central configurations, not antipodal if $m=2$. Possibly restricting the intervals, assume that the function $\|\gamma(\cdot)\|$ is strictly concave or convex in $[\alpha_i,\beta_i]$, for every $i=1, \dots, m$.\\ Then, fixed $i\in\{1, \ldots, m\}$, $\xi_E\in [\alpha_i, \beta_i]$ and $\xi_I\in\bigcup_{j\in NA(j)}[\alpha_j, \beta_j]$, for $h$ large enough the quantities 
		\begin{equation}\label{eq:der_sec_S}
			\partial^2_bS_E(\xi_E, \cdot)+\partial^2_aS_I(\cdot, \xi_I; h), \quad \partial^2_bS_I(\xi_I, \cdot; h)+\partial^2_aS_E(\cdot, \xi_E)
		\end{equation}
	have constant sign in $[\alpha_i,\beta_i]$. 
	\end{lemma}
\begin{proof}
	Let us consider the first quantity in Eq. \eqref{eq:der_sec_S}, and, to fix the ideas, let us suppose that $\bx_i$ is a strict maximum for $\|\gamma(\cdot)\|$. Hence, there exists a constant $A$ such that
	\begin{equation*}
		\frac{d^2}{d\xi^2}\|\gamma(\xi)\|>A>0, \quad \forall\xi\in[\alpha_i,\beta_i]. 
	\end{equation*}
By the differentiable dependence of the solutions of the outer problem with respect to variations of the endpoints, the $C^2$-regularity of $\gamma$ and the compactness of $[\alpha_i, \beta_i]$, one can ensure that there exists a constant $C>0$ such that, for every $\xi_1,\xi_2$ in $[\alpha_i, \beta_i]$, 
\begin{equation*}
	-C\leq \partial^2_b S_E(\xi_1,\xi_2)\leq C; 
\end{equation*}
moreover, from Lemma \ref{lem:lemma_sus} one has that 
\begin{equation*}
	\partial^2_aS_I(\xi, \xi_I; h)=\sqrt{\mathcal E+h}\ \frac{d^2}{d\xi^2}\|\gamma(\xi)\| +\frac{\mu}{\sqrt{\mathcal E+h}}\ \frac{\partial^2}{\partial \xi^2}F_1(\xi,\xi_I; \mathcal E+h), 
\end{equation*}
for every $\xi\in[\alpha_i,\beta_i]$, where $F_1$ is uniformly $C^2$-bounded with respect to the first two variables. One then obtains 
\begin{equation*}
	\partial^2_bS_E(\xi_E, \xi)+\partial^2_aS_I(\xi, \xi_I; h)\geq -C+\sqrt{\mathcal E+h}\ A +\frac{\mu}{\sqrt{\mathcal E+h}}\frac{\partial^2}{\partial \xi^2}F_1(\xi,\xi_I; \mathcal E+h),  
\end{equation*}
and then the thesis follows by the uniform boundedness of $F_1$.  
\end{proof}

From the proof of the previous lemma it follows that the threshold for $h$ can be chosen uniformly also in $\xi_E$ and $\xi_I$. This fact allows to deduce our final result.

\begin{theorem}\label{thm:final}
	Let us suppose that there exist $\bx_1, \dots, \bx_m$ non-degenerate central configurations, not antipodal if $m=2$. Then, if $h$ is large enough, the projection map $\pi$ defined in Theorem \ref{thm:ex_dyn_sym_intro} is also injective. In other words, the dynamics of the refraction billiard admits a topologically chaotic subsystem. 
\end{theorem}
\begin{proof}
	Let us go back to the proof of Proposition \ref{prop: punto critico azione}: to prove that $\pi$ is injective it is sufficient to prove that the critical point found through Poincar\'e-Miranda Theorem is unique. \\
	In the notation of Theorem \ref{thm: miranda}, if the functions $F_k$ are monotone, then the solution of problem \eqref{eq: miranda sistema} must be unique in $R$. This is true if, for any $k\in\{1, \ldots, d\}$, and any fixed $x_1, \ldots, x_{k-1}, x_{k+1}, \ldots, x_d$, the quantity 
	\begin{equation*}
		\frac{\partial}{\partial x_k}F_{k}(x_1, \ldots, x_k)
	\end{equation*}
has constant sign for $x_k\in[-L, L]$. \\
In our case, this is ensured by Lemma \ref{lem:pi_iniettiva}. 
\end{proof}

\section{Final remarks and conclusions}

In this work, along with the previous papers \cite{deblasiterraciniellissi,IreneSusNew}, we presented the analysis of a brand new dynamical model of interest in Celestial Mechanics, starting from the basic study of its fixed points and arriving to its non-integrability. In particular there is \emph{fil rouge} between the first and the present paper: an elliptic domain with its center in the origin satisfies the assumptions of Theorem \ref{thm:final} and thus the associated system is chaotic; this represents the analytical proof of the numerical results shown in paper \cite[Figure 11]{deblasiterraciniellissi}. 

Furthermore, the results of the present paper hold also when we deal with \emph{reflective billiards}: in this case, known in literature as \emph{Kepler billiard}, Keplerian arcs are reflected against the domain's boundary, hence just internal arcs are considered (see \cite{Bol2017}). Taking again the example of an elliptic domain, we deduce that, as far as the singularity is in the center of the ellipse, the reflective system is chaotic as well. This negatively complements the recent results by Takeuchi and Zhao \cite{Lei2021,takeuchi2021conformal, takeuchi2022projective} where they consider an elliptic Kepler billiard with the mass in one of the foci, proving its integrability. Note that an ellipse with focus at the origin does not satisfy the hypotheses of Theorem \ref{thm:ex_dyn_sym_intro}, while it does when moving the gravitational center at the center of the ellipse. The transition from integrability to chaoticity in such elliptic model is the subject of a forthcoming paper (\cite{BaDeb2022}).

We stress that refraction billiards can be used to study the motion of a particle subjected to \emph{any} discontinuous potential: although we consider a galactic model, the techniques we used can be implemented in different contexts, in particular emerging from physical models. We cite for instance \emph{inverse magnetic billiards} studied in \cite{gasiorek2019dynamics}.

\appendix
\section{Jacoby distances and Snell's law}\label{sec:appB}
This appendix is devoted to the definition and the introduction of the main results connected to the Jacobi distances related to the inner and outer dynamics, which are used in Section \ref{sec:symb_dyn}. In particular, the relation between the variational and geometric properties of the inner and outer arcs connecting points on $\partial D$ are investigated. 

\begin{definition}\label{def:Jacobi_dist}
	Let $\xi^E_1, \xi_2^E\in (a,b) \subset [0,L]$, where $(a,b)$ satisfies Definition \ref{def:prop_arco_est} and let $z_E(\cdot; \gamma\left(\xi_1^E\right),\gamma\left(\xi_2^E\right) ):[0, T_E]\to\R^2$ be the unique solution of \eqref{eq:bolzaext_intro}. The \textbf{outer Jacobi distance} between the two points is given by\footnote{Here and in the following, to ease the notation, we will omit the arguments in $T_E(\cdot, \cdot)$ and $T_I(\cdot, \cdot;\cdot)$. }
	\begin{equation*}
		d_E\left(\gamma\left(\xi_1^E\right), \gamma\left(\xi_2^E\right)\right)\eqdef\int_0^{T_E}\| z'_E(s)\|\sqrt{V_E\left(z_E(s)\right)}ds.
	\end{equation*}
	According to Definition \ref{def:prop_arco_int}, let  $\mathcal{A}$ satisfy the inner-arc property, $h>h_0$, and $\xi^I_1, \xi_2^I\in {\mathcal A}$ be such that there exists  a unique {\it (TnT)} solution of \eqref{eq: problema interno}, $z_I(\cdot; \gamma\left(\xi_1^I\right),\gamma\left(\xi_2^I\right); h ):[0, T_I]\to\R^2$. The \textbf{inner Jacobi distance} between the endpoints is then given by 
	\begin{equation*}
		d_I\left(\gamma\left(\xi_1^I\right), \gamma\left(\xi_2^I\right); h\right)\eqdef\int_0^{T_I}\| z'_I(s)\|\sqrt{V_I\left(z_I(s)\right)}ds.
	\end{equation*}
\end{definition}
The integral quantities in the above definitions can be interpreted as the \emph{Jacobi lengths} of $z_E$ (resp. $z_I$) related to the outer (resp. inner) potential (see also \cite{soave2014symbolic}). These quantities can indeed be defined in more general frameworks, although it is beyond the scope of our work. 
\begin{remarks} In view of Definition \ref{def:Jacobi_dist}, we observe that: 
	\begin{itemize}
		\item the uniqueness of the solutions of the problems \eqref{eq:bolzaext_intro} and \eqref{eq: problema interno} implies that the quantities $d_E$ and $d_I$ are well defined and differentiable. In particular, when $\xi_1^I=\xi_2^I$ and the inner arc is the collision-ejection solution, the integral representing its Jacobi length can be interpreted in its regularised formulation in the Levi-Civita plane (see for example \cite{deblasiterraciniellissi}); 
		\item the functions $d_E$ and $d_I$ are not proper distances: as a matter of fact, when $\xi_1^E=\xi_2^E$ or $\xi_1^I=\xi_2^I$ they do not vanish but represent the non-zero Jacobi length of the homothetic (outer or inner) arc; 
		\item the Jacobi length is invariant under reparametrizations of the path $z_{E/I}$.  
	\end{itemize}
\end{remarks}
It is useful to express the distances as  functions of the curve's parameter rather than of the endpoints in the plane. Given then  $\xi^E_1, \xi_2^E\in[0,L]$ and $h>0$, $\xi^I_1, \xi_2^I\in[0,L]$ as in Definition \ref{def:Jacobi_dist}, we can define the functions

\begin{equation}\label{eq:defS}
	S_E(\xi_1^E, \xi_2^E)\eqdef d_E(\gamma(\xi_1^E), \gamma(\xi_2^E)), \quad S_I(\xi_1^I, \xi_2^I; h)=d_I(\gamma(\xi_1^I), \gamma(\xi_2^I); h). 
\end{equation} 
We observe that this definition is coherent with the one of \emph{generating function} in classical Birkhoff billiards (see \cite{Tabbook}).\\
Let us recall that the distances $d_E$ and $d_I$ are infinitely-many differentiable as functions of the endpoints in every set in which they are well defined; by the chain rule, this implies that the lengths $S_E$ and $S_I$ inherit the regularity of the curve $\gamma$: in our case, since this curve is supposed to be at least of class $C^1([0,L])$, provided that the energy jump $h$ is large enough the two Jacobi lengths have the same regularity in every subset of $[0,L]\times[0,L]$ in which the inner or  outer dynamics are well defined. \\
Denoting with $\partial_a$ and $\partial_b$ the partial derivatives respectively with respect to the first and second variable, one has  
\begin{equation}\label{eq: partial S}
	\begin{aligned}
		\partial_a S_{E}\left(\xi_1^{E}, \xi_2^{E}\right)=\nabla_{P_1}d_{E}\left(\gamma\left(\xi_1^{E}		\right),\gamma\left(\xi_2^{E}\right)\right)\cdot \dot\gamma\left(\xi_1^{E}\right)\\
		\partial_b S_{E}\left(\xi_1^{E}, \xi_2^{E}\right)=\nabla_{P_2}d_{E}\left(\gamma\left(\xi_1^{E}		\right),\gamma\left(\xi_2^{E}\right)\right)\cdot \dot\gamma\left(\xi_2^{E}\right) 
	\end{aligned}
\end{equation}
(and similarly for $S_I$), where $\nabla_{P_1}$ and $\nabla_{P_2}$ denote respectively the gradient with respect to the first and the second point.
In order to compute this quantities, it is worth a more general digression on the relation between solutions of suitable fixed-ends problem and the critical points of the corresponding Jacobi length. More precisely, let us fix $P_1,P_2\in\R^2$  and consider the Bolza problem 
\begin{equation}\label{eq:BolzaAppB}
	\begin{cases}
		{z''}(s)=\nabla V\left(z(s)\right), \quad &s\in[0, T]\\
		\frac{1}{2}\| z'(s)\|^2-V\left(z(s)\right)=0 &s\in[0, T]\\
		z(0)=P_1, \quad z(T)=P_2
	\end{cases}
\end{equation}  
for some $T>0$, where $V$ is a $C^2$ potential defined on an open set $\Omega\subseteq\R^2$. The corresponding Jacobi length is given by 
\begin{equation}\label{eq:Jacobilength}
	L(z)\eqdef\int_{0}^T\|z'(s)\|\sqrt{V(z(s))}ds, 
\end{equation}
and is defined on the set of the path connecting $P_1$ to $P_2$ and belonging to  the Hill's region associated to $V$ 
\begin{equation*}
	H_{P_1P_2}\eqdef\left\{z\in H^1\left([0,T], \R^2\right)\text{ }|\text{ }z(0)=P_1, z(T)=P_2, V(z(s))\geq0\text{ for all }s\in[0, T]\right\}. 
\end{equation*}
Let us start by stating some classical results coming from Critical Points theory. 
\begin{lemma}\label{lem: appBEL}
	A path $z\in H_{P_1P_2}$ is a critical point of $L(\cdot)$ if and only if it is a solution of the Euler-Lagrange equations 
	\begin{equation}\label{eq: EL}
		\frac{d}{ds}\left(\sqrt{V(z(s))}\frac{ z'(s)}{\|z'(s)\|}\right)-\frac{\| z'(s)\|}{2\sqrt{V(z(s))}}\nabla V(z(s))=0 \quad \text{a.e. in }[0,T]. 
	\end{equation}
\end{lemma}
\begin{proof}
	One has that $z$ is a critical point for the Jacobi length if and only if for every $v\in H_0^1([0,T])$ one has that $dL(z)[v]=0$. This is equivalent to require that the following chain of equalities holds:  
	\begin{equation*}
		\begin{aligned}
			0&=\frac{d}{d\epsilon}L(z+\epsilon v)_{|_{\epsilon=0}}=\left(\frac{d}{d\epsilon}\int_{0}^T\|z'(s)+\epsilon v'(s)\|\sqrt{V(z(s)+\epsilon v(s))}\,ds\right)_{|_{\epsilon=0}}\\
			&= \int_0^T\frac{\sqrt{V(z(s))}}{\| z'(s)\|} z'(s)\cdot  v'(s)+\frac{\| z'(s)\|}{2\sqrt{V(z(s))}}\nabla V(z(s))\cdot v(s)\, ds\\
			&=\int_0^T \left[-\frac{d}{ds}\left(\sqrt{V(z(s))}\frac{ z'(s)}{\| z'(s)\|}\right)+\frac{\| z'(s)\|}{2\sqrt{V(z(s))}}\nabla V(z(s))\right]\cdot v(s)\, ds, 
		\end{aligned}
	\end{equation*}
	where in the last equation an integration by parts has been employed. As the identity must be true for every $v\in H_0^1([0,T])$, one has that the Euler-Lagrange equations \eqref{eq: EL} must hold for almost every $s\in[0,T]$.  
\end{proof}
We stress that if $u\in C^1([0,T])\cap H_{P_1P_2}$ is such that $V(z(s))>0$ and $\| z'(s)\|>0$ for every $s\in[0, T]$, by continuity one can infer that Eq. \eqref{eq: EL} holds everywhere in $[0,T]$. 
Making use of the previous result, it is possible to find a connection between solutions of problem \eqref{eq:BolzaAppB} and critical points of $L(\cdot)$.  
\begin{lemma}\label{lem:appB}
	Let $\bar z\in H_{P_1P_2}$ be a solution of problem \eqref{eq:BolzaAppB} such that $V(\bar z(s))>0$ for every $s\in [0,T]$. Then $\bar z$ is also a critical point of the Jacobi length $L$. 
\end{lemma}
\begin{proof}
	Let us start by observing that, if $\bar z$ is a solution of problem \eqref{eq:BolzaAppB} such that $V(\bar z(s))>0$ in $[0,T]$, then $\bar z\in C^2([0, T])\cap H_{P_1P_2}([0, T])$ and $\|{\bar z'}(s)\|>0 $ for every $s\in[0, T]$. We will then prove that $\bar z$ is a critical point for $L(\cdot)$ by verifying that it solves the Euler-Lagrange equations \eqref{eq: EL} for every $s\in[0,T]$. As a matter of fact, one has, for every $s\in[0, T]$,  
	\begin{multline*}
		\frac{d}{ds}\left(\sqrt{V(\bar z(s))}\frac{ {\bar z'}(s)}{\| {\bar z'}(s)\|}\right) \\
		= \frac{\nabla V(\bar z(s))\cdot {\bar z'}(s)}{2\sqrt{V(\bar z(s))}}\frac{{\bar z'}(s)}{\|{\bar z'}(s)\|}+\frac{\sqrt{V(\bar z(s))}}{\|{\bar z'}(s)\|}{\bar z''}(s)-\sqrt{V(\bar z(s))}\frac{{\bar z'}(s)\cdot {\bar z''}(s)}{\|{\bar z'}(s)\|^3}{\bar z'}(s), 
	\end{multline*}
	and the Euler-Lagrange equations follow from the first two lines in Eq. \eqref{eq:BolzaAppB}. The conclusion follows from Lemma \ref{lem: appBEL}.
\end{proof}
Passing now from a generic potential $V$ to the inner and outer potentials of our dynamical system, let us remark that, when $V=V_E$ or $V=V_I$ and $P_1\neq P_2$ are as in Definition \ref{def:Jacobi_dist}, the unique inner or outer solution is in the form described in Lemma \ref{lem:appB}. Starting from this observation, we can now obtain the explicit expressions of the quantities involved in Eq. \eqref{eq: partial S}. In the following, the Jacobi lengths related respectively to the outer and inner potentials will be denoted by $L_E$ and $L_I$. 
\begin{lemma}\label{lem:derivateS}
	Let $\xi_1^E,\xi_2^E,\xi_1^I, \xi_2^I\in [0, L]$ and $h>0$ as in Definition \ref{def:Jacobi_dist}. Then 
	\begin{equation}\label{eq: derivate S}
		\begin{aligned}
			&\partial_aS_{E}(\xi_1^E, \xi_2^E)=-\sqrt{V_{E}(\gamma(\xi_1^E))}\frac{z'_{E}(0; \gamma(\xi_1^E), \gamma(\xi_2^E))}{\|z'_{E}(0; \gamma(\xi_1^E), \gamma(\xi_2^E))\|}\cdot\dot\gamma(\xi_1^E)\\
			&\partial_bS_{E}(\xi_1^E, \xi_2^E)=\sqrt{V_{E}(\gamma(\xi_2^E))}\frac{z'_{E}(T_E; \gamma(\xi_1^E), \gamma(\xi_2^E))}{\|z'_{E}(T_E; \gamma(\xi_1^E), \gamma(\xi_2^E))\|}\cdot\dot\gamma(\xi_2^E)\\
			&\partial_aS_{I}(\xi_1^I, \xi_2^I; h)=-\sqrt{V_{I}(\gamma(\xi_1^I))}\frac{z'_{I}(0; \gamma(\xi_1^I), \gamma(\xi_2^I); h)}{\|z'_{I}(0; \gamma(\xi_1^I), \gamma(\xi_2^I); h)\|}\cdot\dot\gamma(\xi_1^I)\\
			&\partial_bS_{I}(\xi_1^I, \xi_2^I; h)=\sqrt{V_{I}(\gamma(\xi_2^I))}\frac{z'_{I}(T_I; \gamma(\xi_1^I), \gamma(\xi_2^I); h)}{\|z'_{I}(T_I; \gamma(\xi_1^I), \gamma(\xi_2^I), h)\|}\cdot\dot\gamma(\xi_2^I). 
		\end{aligned}	
	\end{equation}
\end{lemma}
\begin{proof}
	Let us observe that the partial derivatives in Eq.  \eqref{eq: partial S} can be expressed as directional derivatives of the (inner or outer) Jacobi length in the direction of $\dot{\gamma}(\xi)$ for suitable $\xi\in[0, L]$. Taking for example $\partial_a S_E(\xi_1^E,\xi_2^E)$ (analogous expressions hold for the other derivatives listed in Eq. \eqref{eq: derivate S}), one has 
	\begin{equation*}
		\partial_aS_E(\xi_1^E, \xi_2^E)=\partial_{1,\dot{\gamma}(\xi_1^E)}d_E\left(\gamma\left(\xi_1^E\right), \gamma\left(\xi_2^E\right)\right)=\partial_{1,\dot{\gamma}(\xi_1^E)}L_E\left(z_E\left(\cdot; \gamma{\left(\xi_1^E\right)}, \gamma{\left(\xi_2^E\right)}\right)\right), 
	\end{equation*}
	where, in general, $\partial_{1, \sigma}$ denotes the directional derivative with respect to variations of the first endpoint in the direction of the unit vector $\sigma$. \\
	To find the explicit expression of $\partial_{a}S_{E}$, let us start by assuming that $\xi_1^E\neq\xi_2^E$. It is straightforward to verify that $z_E(\cdot)\eqdef z_E\left(\cdot; \gamma{\left(\xi_1^E\right)}, \gamma{\left(\xi_2^E\right)}\right)$ is a classical solution of the associated Bolza problem such that $V_E\left(z_E(s)\right)>0$ for every $s\in[0, T_E]$: by Lemma \ref{lem:appB}, the outer arc is a solution of the associated Euler-Lagrange equations for every $s\in[0, T_E]$. {Let us now consider the new parametrization $s\mapsto t(s)$ given by 
		\begin{equation*}
			\begin{cases}
				\displaystyle \frac{dt}{ds}=\frac{\sqrt{2}V_E(z_E(s))}{L_E}\\
				t(0)=0
			\end{cases}
		\end{equation*}
		where $L\eqdef L_E(z_E)\in\R$ as defined in \eqref{eq:Jacobilength} with $V=V_E$, and define $\tilde z(t)\eqdef z_E\left(s(t)\right)$. Defining the new time derivative as  $^\cdot\eqdef\frac{d}{dt}$, it is straighforward to verify that $t\in[0,1]$ and that 
		\begin{equation*}
			\forall t\in[0,1]\quad \|\dot{\tilde z}(t)\|\sqrt{V_E\left(\tilde z(t)\right)}=L_E. 
		\end{equation*}
		Moreover, by the invariance of the Jacobi length under reparametrizations, one has that 
		\begin{equation*}
			S_E\left(\xi_1^E, \xi_2^E\right)=L_E(z_E)=L_E\left(\tilde z\right), 
		\end{equation*}
		and then, by \eqref{eq: partial S}, $\partial_aS_E(\xi_1^E, \xi_2^E)=\partial_{1,\dot{\gamma}(\xi_1^E)}L_E\left(\tilde z\right)$.  
		Starting by \eqref{eq: EL}, one can prove that the reparametrized curve $\tilde z$ satisfies the Euler-Lagrange equations
		\begin{equation*}
			\frac{d}{dt}\left(\frac{\partial \mathcal L}{\partial \dot z}\right)=\frac{\partial \mathcal L}{\partial z}, 
		\end{equation*}
		where $\mathcal L\eqdef \|\dot z(t)\|^2V_E(z(t))=L_E^2$, namely,
		\begin{equation}\label{eq:ELnuovo}
			\frac{d}{dt}\left(2V_E\left(\tilde z(t)\right)\dot{\tilde z}(t)\right)=\|\dot{\tilde z}(t)\|^2\nabla V_E\left(\tilde z(t)\right). 
		\end{equation}
		Let us now compute $\partial_{1,\dot{\gamma}(\xi_1^E)}L_E\left(\tilde z\right)$: differentiating $\mathcal L$ with respect to the first endpoint, one has that 
		\begin{equation}\label{eq:appBfin2}
			\begin{aligned}
				2 L_E \partial_{1,\dot{\gamma}(\xi_1^E)}L_E=\partial_{1,\dot{\gamma}(\xi_1^E)}\mathcal L&=\partial_{1,\dot{\gamma}(\xi_1^E)}\int_0^1\|\dot{\tilde z}(t)\|^2\sqrt{V_E\left(\tilde z(t)\right)}dt=\int_0^1\partial_{1,\dot{\gamma}(\xi_1^E)}\left(\|\dot{\tilde z}(t)\|^2\sqrt{V_E\left(\tilde z(t)\right)}\right)dt\\
				&=\int_0^1 2V_E\left(\tilde z(t)\right)\dot{\tilde z}(t)\cdot\partial_{1,\dot{\gamma}(\xi_1^E)}\dot{\tilde z}(t)+\|\dot{\tilde z}(t)\|^2\nabla V_E\left(\tilde z(t)\right)\cdot\partial_{1,\dot{\gamma}(\xi_1^E)}{\tilde z}(t).
			\end{aligned}
		\end{equation}
		Moreover, multiplying \eqref{eq:ELnuovo} by $\partial_{1,\dot{\gamma}(\xi_1^E)}{\tilde z}(t)$ and integrating the result in $[0,1]$, one obtains 
		\begin{equation}\label{eq:appBfin}
			\begin{aligned}
				&\int_0^1\frac{d}{dt}\left(2V_E\left(\tilde z(t)\right)\dot{\tilde z}(t)\right)\cdot \partial_{1,\dot{\gamma}(\xi_1^E)}{\tilde z}(t) =\int_0^1\|\dot{\tilde z}(t)\|^2\nabla V_E\left(\tilde z(t)\right) \partial_{1,\dot{\gamma}(\xi_1^E)}{\tilde z}(t)\\
				&\Longrightarrow \int_0^1 2V_E\left(\tilde z(t)\right)\dot{\tilde z}(t)\cdot\partial_{1,\dot{\gamma}(\xi_1^E)}\dot{\tilde z}(t)+\|\dot{\tilde z}(t)\|^2\nabla V_E\left(\tilde z(t)\right)\cdot\partial_{1,\dot{\gamma}(\xi_1^E)}{\tilde z}(t)=-2V_E\left(\tilde z(0)\right)\dot{\tilde z}(0)\cdot \dot\gamma\left(\xi_1^E\right), 
			\end{aligned}
		\end{equation}
		where the second equation is obtained by integrating by part and observing that $\partial_{1,\dot{\gamma}(\xi_1^E)}{\tilde z}(0)=\dot{\gamma}(\xi_1^E)$ and $\partial_{1,\dot{\gamma}(\xi_1^E)}{\tilde z}(1)=0$. 
		Comparing now \eqref{eq:appBfin} and \eqref{eq:appBfin2} and recalling the expression of $L_E$, one gets the final expression
		\begin{equation*}
			\partial_{1,\dot{\gamma}(\xi_1^E)} L_E\left(\tilde z\right)=-\frac{V_E\left(\tilde z(0)\right)\dot{\tilde z}(0)}{L_E}\cdot \dot{\gamma}\left(\xi_1^E\right)=-\sqrt{V_E\left(\tilde z(0)\right)}\frac{\dot {\tilde z}(0)}{\|\dot {\tilde z}(0)\|}\cdot\dot{\gamma}\left(\xi_1^E\right). 
		\end{equation*}
		Returning now to the time parameter $s$, one obtains
		\begin{equation*}
			\partial_aS_{E}(\xi_1^E, \xi_2^E)=-\sqrt{V_{E}(\gamma(\xi_1^E))}\frac{z'_{E}(0; \gamma(\xi_1^E), \gamma(\xi_2^E))}{\|z'_{E}(0; \gamma(\xi_1^E), \gamma(\xi_2^E))\|}\cdot\dot\gamma(\xi_1^E). 
		\end{equation*}
	}

	The same identity can be extended to the case $\xi_1^E=\xi_2^E$ by observing that $V_E\left(z_E(s; \gamma(\xi_1^E), \gamma(\xi_1^E))\right)>0$ almost everywhere in $[0, T_E]$ and by taking into account the differentiable dependence of $z_E\left(\cdot; \gamma(\xi_1^E), \gamma(\xi_2^E)\right)$ with respect to variations of the endpoints. \\
	In the inner case one can use the same reasonings, keeping in mind that, whenever a collision occurs, one can consider the corresponding regularized system. 
\end{proof}

We are now ready to validate the refraction Snell's law stated in Eq. \eqref{eq:SnellIntro} to rule the junction between outer and inner arcs and given by 
\begin{equation}\label{eq: snell seni appendice}
	\sqrt{V_E(\gamma(\xi))}\sin\alpha_E=\sqrt{V_I(\gamma(\xi))}\sin\alpha_I, 
\end{equation} 
where:
\begin{itemize}
	\item $\gamma(\xi)$ is the transition point between the outer and inner region (or viceversa); 
	\item $\alpha_E$, $\alpha_I$ are the angles of the two arcs with respect to the outward-pointing normal unit vector to $\gamma$ in $\xi.$
\end{itemize} 
As already pointed out in the Introduction, Eq. \eqref{eq: snell seni appendice} is justified by a variational argument, which can be explicited by means of the partial derivatives of $S_E$ and $S_I$. \\
Let us start by fixing $h>h_0$, where $h_0$ is defined in Definition \ref{def:prop_arco_int};  now take $\xi_E, \xi_I\in[0, L]$ such that there exists a concatenation outer-inner arc starting from $\gamma(\xi_E)$ and arriving in $\gamma(\xi_I)$. More precisely, this means that there exists $\xi\in[0, L]$ such that there exists a unique outer arc between $\gamma(\xi_E)$ and $\gamma(\xi)$ and an unique {\it(TnT)} inner arc from $\xi$ to $\xi_I$.  \\
Although this reasoning can be treated by referring to a more general setting, let us suppose to have a disjoint union of intervals $\mathcal A=\bigcup_{i=1, \dots, m}(a_i, b_i)\subset[0,L]$ satisfying the inner-arc property (Definition \ref{def:prop_arco_int}) and such that every connected component $(a_i, b_i)$, $i=1, \ldots, m$, enjoys the outer-arc property (Definition \ref{def:prop_arco_est}); suppose now that $\xi_E\in(a_i, b_i)$ and $\xi_I\in(a_j,b_j)$, $j\in NA(i)$, where $NA(i)$ has been defined in Eq. \eqref{eq:def_not_antipodal}. 
We have then that for every $\xi\in(a_i,b_i)$ there are exactly one outer arc connecting $\gamma(\xi_E)$ to $\gamma(\xi)$ and one {\it (TnT)} inner arc starting from $\gamma(\xi)$ and 	arriving in $\gamma(\xi_I)$. We can then consider the total Jacobi length of the concatenation composed by these arcs, denoted, following the notation introduced in  Eq. \eqref{eq:defS}, by $S_E(\xi_E, \xi)+S_I(\xi, \xi_I; h)$. 
\begin{definition}
	We say that the concatenation between $\gamma(\xi_E)$ and $\gamma(\xi_I)$ of transition point $\gamma(\bx)$, with $\bx\in[0, L]$, satisfies the Snell's law if $\bx$ is a critical point\footnote{This criticality argument can be replaced in a minimality argument if we consider a more general definition for the inner and outer Jacobi distances, whose variables can be points not necessarily lying on $\partial D$, and restrict our analysis to a strongly convex neighborhood for both the inner and outer Jacobi metric (see \cite{do2016differential}). In such case, indeed, the geodesic arcs connecting any two points are minimizers for the Jacobi length. } for the total Jacobi length, namely, if 
	\begin{equation}\label{eq:SnellApp}
		\frac{d}{d\xi}\left(S_E(\xi_E, \xi)+S_I(\xi, \xi_I; h)\right)_{|_{\xi=\bx}}=0. 
	\end{equation} 
\end{definition}
An analogous condition can be established for a concatenation starting with an inner arc. 

Note that condition \eqref{eq:SnellApp} is equivalent to require that 
\begin{equation}\label{eq:Snell1}
	\partial_bS_E(\xi_E, \bx)+\partial_aS_I(\bx, \xi_I; h)=0, 
\end{equation}
which, in view of Lemma \ref{lem:derivateS}, can be rephrased as 
\begin{equation}\label{eq:Snell2}
	\sqrt{V_E(\gamma(\bx))}\frac{ z'_E(T_E; \gamma(\xi_E), \gamma(\bx))}{\| z'_E(T_E; \gamma(\xi_E), \gamma(\bx))\|}\cdot\dot\gamma(\bx)=\sqrt{V_I(\gamma(\bx))}\frac{ z'_I(0; \gamma(\bx), \gamma(\xi_I); h)}{\| z'_I(0; \gamma(\bx), \gamma(\xi_I); h)\|}\cdot\dot\gamma(\bx). 
\end{equation}
Eq. \eqref{eq:Snell1} is used as an admissibility criterion in Section \ref{sec:symb_dyn} to establish if a concatenation is a trajectory for the complete dynamics. As for Eq. \eqref{eq:Snell2}, it can be easily translated into Eq. \eqref{eq: snell seni appendice}, and can be interpreted as a conservation law for the tangential component of the trajectory's velocity vector across the interface. 

Let us conclude by observing that Eq. \eqref{eq: snell seni appendice} has an evident correlation with the classical Snell's law for straight light rays, which can be derived again from a variational minimization problem (known as \emph{Fermat's principle}). In this sense, our refraction law can be interpreted as a generalization for generic potentials and curved geodesics of this classical Snell's law; in particular, while in the classical case the geodesic arcs are always minimizers for the Jacobi length, in our case the solutions of the Bolza problems are only critical points of the corresponding lengths $L_E(\cdot)$ and $L_I(\cdot)$: this justifies the use of a criticality condition rather than a minimality one (which, in any case, can be retrieved working locally around the transition point).

\section{Existence and properties of Keplerian hyperbol\ae ~connecting two points}\label{sec: appA}

In this appendix we analyze some properties of Keplerian hyperbol\ae\, connecting two not antipodal points on the boundary $\partial D$. The next result is a particular case of some results provided in paper \cite{bolotin2000periodic} (see Lemma 4.1 and Proposition 6.1) where the authors consider much more general singular dynamical systems. The same authors indeed remarked that the classical Keplerian problem is a special case of their construction that can be computed explicitly.

\begin{lemma}\label{lem:lemma_sus}
	Let $p_0,p_1\in\partial D$ such that they are not antipodal, and fix $M>0$. Then for every $E>0$ there exists a unique \tnt\ (possibly regularized) solution of 
	\begin{equation*}\label{eq: problema interno_app}
		\begin{cases}
			\displaystyle	z''(s)=-M\frac{z(s)}{\|z(s)\|^3},\quad &s\in[0,T],\\
			\displaystyle	\frac{1}{2}\|z'(s)\|^2- E-\frac{M}{\|z(s)\|}=0, \quad &s\in[0,T],\\
			z(0)=p_0, \text{ }z(T)=p_1,
		\end{cases}
	\end{equation*}
	called $z_I(\cdot; p_0,p_1; E)$. Moreover, 
	\begin{equation*}
		L_I(z_I)=\sqrt{E}\left(\|p_0\|+\|p_1\|\right) +\frac{M}{\sqrt{E}}\left(F(p_0,p_1; E)-\log\left(\frac{M}{2E}\right)\right), 
	\end{equation*}
	where $F$ is $C^2$-bounded uniformly with respect to $p_0,p_1$ as $E\to\infty$. 
\end{lemma}

\begin{proof}
	The existence and uniqueness of $z_I$ is a classical result in Celestia Mechanics (see for instance \cite[p. 274]{battin}). Let us then consider the Jacobi length of $z_I$
	\begin{equation*}
		L_I(z_I)=\int_0^T\|z_I'(s)\|\sqrt{E+\frac{M}{\|z_I(s)\|}}ds=\sqrt{E}\int_0^T\|z'_I(s)\|\sqrt{1+\frac{M}{E}\frac{1}{\|z_I(s)\|}}ds, 
	\end{equation*}
	and the equivalent Kepler problem with energy equal to $1$ and mass $M/E$.	
	Using Levi-Civita regularization, and in particular considering the transformation (in complex notation) 
	\begin{equation*}
		\frac{d}{ds}=\frac{1}{\sqrt{2}\|z\left(s(\tau)\right)\|}\frac{d}{d\tau}, \qquad 
		w^2(\tau) = z(s(\tau))
	\end{equation*}
	one arrives to the harmonic repulsor-system for a suitable $\tilde T$ 
	\begin{equation}\label{eq:prob_LC}
		\begin{cases}
			w''(\tau)=w(\tau)&\quad \tau\in[0, \tilde T]\\
			\frac{1}{2}\|w'(\tau)\|^2-\frac12\|w(\tau)\|^2=E' &\quad\tau\in[0, \tilde T]\\
			w(0)=w_0,\ w(\tilde T)=w_1
		\end{cases}
	\end{equation}
	where $E'=M/2E$ and $w_0,w_1\in\mathbb C$ are such that $w_0^2=p_0$ and $w_1^2=p_1$. Since the complex square determines a double covering over $\C$, we have two possible choices for every $w_i$, $i=0,1$. To obtain solution corresponding to the \tnt arc in the physical Kepler problem, we need to take $w_0=\sqrt{p_0}$ and $w_1=-\sqrt{p_1}$ (or, equivalently, $w_0=-\sqrt{p_0}$ and $w_1=\sqrt{p_1}$); with some geometric remark we verify that, with this choice, $w_0\cdot w_1<0$. 
	Back to $\R^2$, the solution of \eqref{eq:prob_LC} is
	\begin{equation*}
		w(\tau)=Ae^\tau+B e^{-\tau}, 
	\end{equation*}
	where the vectors $A,B\in\R^2$ can be expressed in terms of $w_0,w_1$ and $\tilde T$ as
	\begin{equation*}
		A=\frac{w_1-w_0 e^{-\tilde T}}{\vertspace e^{\tilde T}-e^{-\tilde T}}, \quad B=\frac{w_0 e^{\tilde T}-w_1}{\vertspace e^{\tilde T}-e^{-\tilde T}}, 
	\end{equation*}
	and the energy conservation law gives $E'=-2 A\cdot B$. 

Computing the Jacobi length of $z_I$,  one has then that (see \cite{deblasiterraciniellissi} for more explicit computations)
\begin{equation*}
	\begin{aligned}
		L_I(z_I)&=\sqrt{E}\int_0^T\|z'_I(s)\|\sqrt{1+\frac{M}{E}\frac{1}{\|z_I(s)\|}}ds=2\sqrt{E}\int_0^{\tilde T}\|w'(\tau)\|\sqrt{2 E'+\|w(\tau)\|^2}d\tau\\
		&=2\sqrt{E}\left[\frac12\frac{e^{\tilde T}+e^{-\tilde T}}{e^{\tilde T}-e^{-\tilde T}}(\|w_0\|^2+\|w_1\|^2)-\frac{2 w_0\cdot w_1}{e^{\tilde T}-e^{-\tilde T}}+\tilde T E'\right]. 
	\end{aligned}
\end{equation*}

It is then necessary to express all the quantities depending on $\tilde T$ in terms of $w_0$ and $w_1$. Let us start by computing 
\begin{equation*}
	y(w_0,w_1; E')\eqdef\frac{e^{\tilde T}+e^{-\tilde T}}{2}=-\frac{w_0\cdot w_1}{2E'}+\sqrt{\frac{1}{2E'}(\|w_0\|^2+\|w_1\|^2)+\left(\frac{w_0\cdot w_1}{2E'}\right)^2+1}; 
\end{equation*}
since $w_0\cdot w_1<0$, the function $1/y$ is analytic on a right neighborhood of $E'=0$ (which indeed corresponds to $E \to +\infty$). Using the expression of $y$, one can then find 
\begin{equation}\label{eq:g1}
	\begin{aligned}
		\tilde T&=\log \left(y+\sqrt{y^2-1}\right)\\
		&=\log\left(\frac{|w_0\cdot w_1|}{2}\right)-\log E'+\log\left(1+\sqrt{1+2E'\frac{\|w_0\|^2+\|w_1\|^2}{(w_0\cdot w_1)^2}+\left(\frac{2 E'}{w_0\cdot w_1}\right)^2}\right)\\
		&\ +\log\left(1+\sqrt{1-\frac{1}{y^2}}\right)\\
		&=-\log E'+g_1(w_0,w_1; E'), 
	\end{aligned}
\end{equation}
where $g_1$ is an analytic function.\\
Furthermore,  we have that 
\begin{equation}\label{eq:g2}
	\begin{aligned}
		\frac{2}{\vertspace e^{\tilde T}-e^{-\tilde T}}&=\frac{1}{\sqrt{y^2-1}}	\\
		&= \frac{\sqrt{2}}{|w_0\cdot w_1|}E'\left(1+\frac{E'(\|w_0\|^2+\|w_1\|^2)}{(w_0\cdot w_1)^2}+\sqrt{1+\frac{2 E'(\|w_0\|^2+\|w_1\|^2)}{(w_0\cdot w_1)^2}+\frac{4{E'}^2}{(w_0\cdot w_1)^2}}\right)^{-1/2}\\
		&= E' g_2(w_0,w_1; E'), 
	\end{aligned}
\end{equation} 
where again $g_2$ is analytic. As for the coefficient of $(\|w_0\|^2+\|w_1\|^2)$ one has that 
\begin{equation*}
	\frac{e^{\tilde T}+e^{-\tilde T}}{e^{\tilde T}-e^{-\tilde T}}=\frac{1}{\sqrt{1-\frac{1}{y^2}}}; 
\end{equation*}
this term is then analytic and its Taylor expansion around $E'=0$ is of the form 
\begin{equation}\label{eq:g3}
	\frac{e^{\tilde T}+e^{-\tilde T}}{e^{\tilde T}-e^{-\tilde T}}=1+E' g_3(w_0,w_1; E'). 
\end{equation}
Taking together Eqs. \eqref{eq:g1}, \eqref{eq:g2} and \eqref{eq:g3}, and recalling that $E'=M/2E$, one obtains
\begin{equation*}
	\begin{aligned}
		L_I(z_I)&=2\sqrt{E}\left(\frac12\left(\|w_0\|^2+\|w_1\|^2\right)+E' G(w_0,w_1; E')-E'\log E'\right)\\  &=\sqrt{E}\left(\|p_0\|+\|p_1\|\right)+\frac{M}{\sqrt{E}}\left(F(p_0, p_1; E)-\log\left(\frac{M}{2E}\right)\right), 
	\end{aligned}
\end{equation*}
where $F$ is $C^2$-bounded with respect to $p_0, p_1$ as $E\to\infty$. 
\end{proof}

\section*{Data availability statement}
This manuscript contains only theoretical elaborations. No data are associated with this research.

\bibliography{vivinabibliog}
\bibliographystyle{abbrv}
\end{document}